\documentclass{birkjour}
\usepackage{amsfonts,amssymb,amsmath,amsgen,amsopn,amsbsy}
 \usepackage{amssymb,stmaryrd}
\usepackage{amsmath}
\usepackage{graphicx}
\usepackage{epstopdf}
\usepackage{xcolor}

 \numberwithin{equation}{section}
\newtheorem{theorem}{Theorem}[section]

\newtheorem{lemma}[theorem]{Lemma}
\newtheorem{proposition}[theorem]{Proposition}

\theoremstyle{definition}
\newtheorem{definition}[theorem]{Definition}

\theoremstyle{remark}
\newtheorem{remark}[theorem]{Remark}

\begin{document}

%
%
%
%
%
%
%
%
%

\setlength{\headheight}{26pt}
\title[Positive Toeplitz Operators]
 {Positive Toeplitz Operators from \\ a Harmonic Bergman-Besov Space \\ into Another}

\author[\"{O}mer Faruk Do\u{g}an]{\"{O}mer Faruk Do\u{g}an }



\date{January 1, 2004}
\address{Department of Mathematics, Tek$\dot{\hbox{\i}}$rda\u{g} \\ Namik Kemal University,
59030 \\ Tek$\dot{\hbox{\i}}$rda\u{g}, Turkey,}

\email{ofdogan@nku.edu.tr}

\begin{abstract}
We define positive Toeplitz operators between harmonic Berg\-man-Be\-sov spaces $b^p_\alpha$ on the unit ball of $\mathbb{R}^n$ for the full ranges of parameters $0<p<\infty$, $\alpha\in\mathbb{R}$. We give characterizations of bounded and compact Toeplitz operators taking one harmonic  Bergman-Besov space  into another in terms of Carleson and vanishing Carleson measures. We also give characterizations for a positive Toeplitz operator on $b^{2}_{\alpha}$    to be a Schatten class operator $S_{p}$  in terms of averaging functions and Berezin transforms for $1\leq p<\infty$, $\alpha\in\mathbb{R}$. Our results extend those known for harmonic weighted Bergman spaces.
\end{abstract}

\maketitle

\section{Introduction}\label{s-introduction}

Let $n\geq 2$ be an integer and $\mathbb{B}=\mathbb{B}_{n}$ be the open unit ball in $\mathbb{R}^n$. Let $\nu$  be the Lebesgue volume  measure on $\mathbb{B}$ normalized so that $\nu(\mathbb{B})=1$.
For  $\alpha\in \mathbb{R}$, we define the weighted volume measures $\nu_\alpha$ on $\mathbb{B}$ by
\[
d\nu_\alpha(x)=\frac{1}{V_\alpha} (1-|x|^2)^\alpha d\nu(x).
\]
These measures are finite when $\alpha>-1$ and in this case we choose $V_\alpha$ so that $\nu_\alpha(\mathbb{B})=1$. Naturally $V_0=1$.  For $\alpha\leq -1$, we set $V_\alpha=1$. We denote the Lebesgue classes with respect to $\nu_\alpha$ by $L^p_{\alpha}$,  $0<p<\infty$, and the corresponding norms by $\|\cdot\|_{L^p_{\alpha}}$.

Let $h(\mathbb{B})$ be the space of all complex-valued harmonic functions on $\mathbb{B}$ with the topology of uniform convergence on compact subsets. The space of bounded harmonic functions on $\mathbb{B}$ is denoted by $h^{\infty}$.
For $0<p<\infty$ and $\alpha>-1$, the harmonic weighted Bergman space $b^p_\alpha$ is defined by $b^p_\alpha=  L^p_\alpha \cap h(\mathbb{B})$ endowed with the norm
$\|\cdot\|_{L^p_{\alpha}}$.  The subfamily $b^2_\alpha$ is a  reproducing kernel Hilbert space with respect to the inner product $[f,g]_{b^2_\alpha}=\int_{\mathbb{B}}f\overline{g} \, d\nu_{\alpha}(x)$ and with the reproducing kernel $R_\alpha(x,y)$  such that $f(x)=[f,R_\alpha(x,\cdot)]_{b^2_\alpha}$ for every $f\in b^2_\alpha$ and $x\in \mathbb{B}$. It is well-known that $R_\alpha$ is real-valued  and $R_\alpha(x,y)=R_\alpha(y,x)$. The homogeneous expansion of $R_\alpha(x,y)$ is given in the $\alpha>-1$ part of the formulas (\ref{Rq - Series expansion}) and (\ref{gamma k q-Definition}) below (see \cite{DS}, \cite{GKU2}).

The theory of Toeplitz operators on harmonic Bergman spaces on the unit ball is a well established subject. Harmonic weighted Bergman space $b^2_\alpha$ with $\alpha>-1$ is naturally imbedded in  $  L^2_\alpha$ by the inclusion $i$.
For $\alpha>-1$, the orthogonal projection  $ Q_\alpha:L^2_\alpha \to b^2_\alpha$ is given by the integral operator

\begin{equation}\label{orthogonal projection}
  Q_\alpha f(x)= \frac{1}{V_\alpha} \int_{\mathbb{B}} R_\alpha(x,y) f(y) (1-|y|^2)^\alpha d\nu(y) \quad (f\in L^2_\alpha ).
\end{equation}
This integral operator plays a major role in the theory of weighted harmonic
Bergman spaces and the question when the Bergman projection $ Q_\alpha:L^p_\beta \to b^p_\beta$  is bounded is studied in many sources such as
(\cite[Theorem 3.1]{JP}, \cite[Theorem 2.5]{PE}, \cite[Theorem 3.1]{KS}).  Then one defines the Toeplitz operator  ${_{\alpha}}T_{\phi} :b^p_\alpha \to b^p_\alpha$ with symbol $\phi$ by ${_{\alpha}}T_{\phi}=Q_{\alpha}M_{\phi}i$, where $M_{\phi}$ is the operator of multiplication by $\phi$. Let $\mu$ be a finite complex Borel measure on $\mathbb{B }$. The Toeplitz operator ${_{\alpha}}T_{\mu}$ with symbol $\mu$ is defined by
\[
{_{\alpha}}T_{\mu}f(x) = \int_{\mathbb{B}} R_\alpha(x,y)  f(y)  d\mu(y)
\]
for $f \in h^{\infty}$. The operator ${_{\alpha}}T_{\mu}$ is more general and reduces to ${_{\alpha}}T_{\phi}$  when $d\mu=\phi d\nu_{\alpha}$. Especially, positive symbols of bounded and compact Toeplitz operators are completely characterized in term of Carleson measures as in \cite{M},\cite{M2} on the ball and in \cite{CLN1} on smoothly bounded domains; these results are all concerned with Toeplitz operators from a harmonic Bergman space into itself. Toeplitz operators from a harmonic Bergman space into another are considered and positive symbols of bounded and compact Toeplitz operators are characterized in \cite{CLN2} on smoothly bounded domains and in \cite{CKY} on the half space.

The weighted harmonic Bergman spaces $b^p_\alpha$ initially defined for $\alpha > -1$
can be extended to the whole range $\alpha \in \mathbb{R}$. These are studied in detail in \cite{GKU2}. We call the extended family $b^p_\alpha$ $(\alpha \in \mathbb{R}) $  harmonic Bergman-Besov spaces and the corresponding reproducing kernels $R_\alpha(x,y)$ $(\alpha \in \mathbb{R})$  harmonic Bergman-Besov kernels. The homogeneous expansion of  $R_\alpha(x,y)$ can be expressed in terms of zonal harmonics
\begin{equation}\label{Rq - Series expansion}
R_\alpha(x,y)=\sum_{k=0}^{\infty} \gamma_k(\alpha) Z_k(x,y) \quad  (\alpha\in \mathbb{R}, \, x,y\in \mathbb{B}),
\end{equation}
where (see \cite[Theorem 3.7]{GKU1}, \cite[Theorem 1.3]{GKU2})

\begin{equation}\label{gamma k q-Definition}
        \gamma_k(\alpha):= \begin{cases}
         \dfrac{(1+n/2+\alpha)_k}{(n/2)_k}, &\text{if $\, \alpha > -(1+n/2)$}; \\
         \noalign{\medskip}
         \dfrac{(k!)^2}{(1-(n/2+\alpha))_k (n/2)_k}, &\text{if $\, \alpha \leq -(1+n/2)$},
\end{cases}
\end{equation}
and $(a)_b$ is the Pochhammer symbol. For definition and details about $Z_k(x,y)$, see \cite[Chapter 5]{ABR}.

The spaces $b^p_\alpha$  can  be defined by using the radial differential operators $D^t_s$ $(s,t \in \mathbb{R})$ introduced in \cite{GKU1} and \cite{GKU2}. These operators are defined in terms of reproducing kernels of harmonic Besov spaces and are specific to these spaces, but still mapping  $h(\mathbb{B})$ onto itself.  The properties of $D^t_s$ will be reviewed in Section \ref{s-preliminaries}. Consider the linear transformation $I_{s}^{t}$ defined for $f\in h(\mathbb{B})$ by

\begin{equation*}
  I^t_s f(x) := (1-|x|^2)^t D^t_s f(x).
\end{equation*}

\begin{definition}\label{definition of the h B-B space}
For $0<p<\infty$ and $\alpha \in \mathbb{R}$, we define the harmonic Bergman-Besov space $b^p_\alpha$ to consist of all $f\in h(\mathbb{B})$ for which $ I^t_s f$
belongs to  $L^p_\alpha$ for some  $s,t$ satisfying (see \cite{GKU2} when $1\leq p<\infty$, and \cite{DOG} when $0<p<1$)

\begin{equation}\label{alpha+pt}
 \alpha+pt>-1.
 \end{equation}
The quantity
\[
\|f\|^p_{b^p_\alpha} = \| I^t_s f\|^p_{L^p_\alpha}=\frac{1}{V_\alpha}\int_{\mathbb{B}} |D^t_s f(x)|^p (1-|x|^2)^{\alpha+pt} d\nu(x) <\infty
\]
defines a norm (quasinorm when $0<p<1$) on $b^p_\alpha$ for any such $s,t$.
\end{definition}
It is well-known that the above definition is independent of $s,t$ under (\ref{alpha+pt}),  and the norms (quasinorms when $0<p<1$) on a given space are all equivalent. Thus for a given pair $s,t$,  $ I^t_s$ isometrically imbeds $b^p_\alpha$ into $L^p_\alpha$ if and only if (\ref{alpha+pt}) holds.

Strictly speaking, the norm (quasinorm when $0<p<1$) depends on $s$ and $t$ but this is not mentioned as it is known that every choice of the pair $(s,t)$ leads to an equivalent norm. Harmonic Bergman-Besov projections $Q_{s}$ that map Lebesgue classes boundedly onto Bergman-Besov spaces $b^p_\alpha$ can be precisely identified as  in the case of harmonic weighted Bergman spaces by
\begin{equation}\label{two}
\alpha+1<p(s+1).
\end{equation}
Then $I_{s}^{t}$ is a right inverse to $Q_{s}$. This is all done in \cite{GKU2}.

Now let $\alpha\in \mathbb{R}$,  s and t satisfing (\ref{two}) and (\ref{alpha+pt}), and a measurable function $\phi$ on $\mathbb{B}$ be given. Harmonic Bergman-Besov projections $Q_{s}$ forces us to define Toeplitz operators on all $b^p_\alpha $ as follows.  We define the Toeplitz operator ${_{s,t}}T_{\phi}: b^p_\alpha \to b^p_\alpha $ with symbol $\phi$ by ${_{s,t}}T_{\phi}=Q_{s}M_{\phi}I_{s}^{t}$. Explicitly,
\[
{_{s,t}}T_{\phi}f(x) = \int_{\mathbb{B}} R_s(x,y) \phi(y) I^t_s f(y)  d\nu_{s}(y) \quad (f\in b^p_\alpha).
\]
We see that ${_{s,t}}T_{\phi}$ makes sense if $\phi \in L^{1}_{s+t}$ and $f$ is a harmonic polynomial. Hence ${_{s,t}}T_{\phi}$ is a densely defined on $b^p_\alpha $ for such $\phi$, because harmonic polynomials are dense in each $b^p_\alpha $.
When $\alpha>-1$ and $p>1$, one can choose $t=0$ and a value of $s$ satisfying (\ref{two})  is $s=\alpha$. Then $I_{\alpha}^{0}$ is inclusion, and  ${_{s,t}}T_{\phi}$  reduces to the classical Toeplitz operator ${_{\alpha}}T_{\phi}=Q_{\alpha}M_{\phi}i$ on the harmonic weighted Bergman spaces  $b^p_\alpha $. We use the term  classical to mean a Toeplitz operator with $i=I_{\alpha}^{0}$. The value $s=\alpha$ does not work when $\alpha \leq -1$. It is possible to take $s \neq \alpha$ also when $\alpha>-1$. So we have more general Toeplitz operators defined via $I_{s}^{t}$  strictly on harmonic Bergman spaces too. It turns out that the properties of Toeplitz operators studied in this paper are independent of $s,t$ under (\ref{two}) and (\ref{alpha+pt}).

Having obtained the integral form for ${_{s,t}}T_{\phi}$, we can now define Toeplitz operators on $b^p_\alpha $ with symbol $\mu$. Let $\alpha$, and $s$ and $t$ satisfing (\ref{two}) and (\ref{alpha+pt}) be given. We define
\[
{_{s,t}}T_{\mu}f(x) = \frac{V_\alpha}{V_{s}} \int_{\mathbb{B}} R_s(x,y) I^t_s f(y) (1-|y|^{2})^{s-\alpha}d\mu(y) \quad (f\in b^p_\alpha).
\]
The operator ${_{s,t}}T_{\mu}$ is more general and reduces to ${_{s,t}}T_{\phi}$ when $d\mu=\phi d\nu_{\alpha}$.   It makes sense when
\[
d\kappa(y)=(1-|y|^{2})^{s+t-\alpha}d\mu(y)
\]
is finite  and $f$ is a harmonic polynomial. Like ${_{s,t}}T_{\phi}$,  ${_{s,t}}T_{\mu}$ is  densely defined on $b^p_\alpha $ for  finite $\kappa$. Note that $\mu$ need not be finite in conformity with that $\alpha$ is unrestricted.

In this paper, we consider the Toeplitz operator ${_{s,t}}T_{\mu}$ with positive symbol and characterize those that are bounded and compact from  a harmonic Bergman-Besov space $b^{p_{1}}_{\alpha_{1}}$ into another $b^{p_{2}}_{\alpha_{2}}$ for $0<p_{1},p_{2}<\infty$ and $\alpha_{1},\alpha_{2} \in \mathbb{R}$. Our main tools are Carleson measures and Berezin transforms.  Let $\mu$ be a positive Borel measure $\mu$ on $\mathbb{B}$. For $\lambda>0$ and $\alpha>-1$, we say that $\mu$ is a   $(\lambda,\alpha)$-Bergman-Carleson measure if for any two positive numbers $p$ and $q$ with $q/p=\lambda$, the inclusion $i: b^p_\alpha\to L^q(\mu)$ is bounded, that is, if
\[
\left(\int_{\mathbb{B}} |f(x)|^q\, d\mu(x) \right)^{1/q} \lesssim \|f\|_{b^p_\alpha}, \qquad (f\in b^p_\alpha).
\]
We can now state our main result.

\begin{theorem}\label{Theorem-1}
Let $0<p_{1},p_{2}<\infty$ and $\alpha_{1},\alpha_{2} \in \mathbb{R}$. Suppose that $\alpha_{1}+p_{1}t>-1$, $\alpha_{2}+p_{2}t>-1$ and

\begin{equation} \label{Equ-1}
n+s+1>n\max\left(1,\frac{1}{p_{i}}\right)+\frac{1+\alpha_{i}}{p_{i}}, \qquad i=1,2.
\end{equation}
Let
\[
\zeta=1+\frac{1}{p_{1}}-\frac{1}{p_{2}}, \qquad \gamma=\frac{1}{\zeta}\left(s+t+\frac{\alpha_{1}}{p_{1}}-\frac{\alpha_{2}}{p_{2}}\right).
\]
Let $\mu$ be a positive Borel measure on $\mathbb{B}$ and $d\kappa(y)=
(1-|x|^{2})^{s+t-\alpha_{1}}d\mu(y)$. Then the following statements are equivalent:
\begin{enumerate}
\item[(i)] ${_{s,t}}T_{\mu}$ is bounded  from  $b^{p_{1}}_{\alpha_{1}}$ to $b^{p_{2}}_{\alpha_{2}}$.
\item[(ii)] $\kappa$ is a $(\zeta,\gamma)$-Bergman-Carleson measure.
\end{enumerate}
\end{theorem}

\begin{remark}
In Theorem \ref{Theorem-1}, the condition
\begin{equation}\label{three}
n+s+1>n\max\left(1,\frac{1}{p_{1}}\right)+\frac{1+\alpha_{1}}{p_{1}}
\end{equation}
is used to prove that (i) implies (ii), whereas the condition
\begin{equation}\label{four}
n+s+1>n\max\left(1,\frac{1}{p_{2}}\right)+\frac{1+\alpha_{2}}{p_{2}}
\end{equation}
is needed to prove that (ii) implies (i). Moreover, when $p_{1}\geq 1$, condition (\ref{three}) reduces to $\alpha_{1}+1<p_{1}(s+1)$, which is equivalent to the fact that $Q_{s}$ is bounded from $L^{p_{1}}_{\alpha_{1}}$ onto $b^{p_{1}}_{\alpha_{1}}$. In a similar way when $p_{2}\geq 1$, condition (\ref{four})  is equivalent to the fact that $Q_{s}$ is bounded from $L^{p_{2}}_{\alpha_{2}}$ onto $b^{p_{2}}_{\alpha_{2}}$.
\end{remark}

 In order to characterize compact positive Toeplitz operators ${_{s,t}}T_{\mu}$ from  harmonic Bergman-Besov spaces $b^{p_{1}}_{\alpha_{1}}$ into another $b^{p_{2}}_{\alpha_{2}}$ for all $0<p_{1},p_{2}<\infty$ and $\alpha_{1},\alpha_{2} \in \mathbb{R}$, we introduce the notion of vanishing $(\lambda,\alpha)$-Bergman-Carleson measures. We say that $\mu\geq 0$ is a vanishing $(\lambda,\alpha)$-Bergman-Carleson measure if for any two $0<p,q<\infty$ satisfying $q/p=\lambda$ and any sequence $\{f_{k}\}$ in $b^{p}_{\alpha}$ with $f_{k}\to 0$ uniformly on each compact subset of $\mathbb{B}$ and $ \|f_{k}\|_{b^{p}_{\alpha}}\leq 1$,

\[
\lim_{k\to \infty}\int_{\mathbb{B}} |f_{k}(x)|^q\, d\mu(x) =0.
\]

\begin{theorem}\label{Theorem-2}
Let $0<p_{1},p_{2}<\infty$ and $\alpha_{1},\alpha_{2} \in \mathbb{R}$. Let $s,t,\zeta$, and $\gamma$ be as in Theorem \ref{Theorem-1}.
Let $\mu$ be a positive Borel measure on $\mathbb{B}$ and $d\kappa(y)=
(1-|x|^{2})^{s+t-\alpha}d\mu(y)$. Then the following statements are equivalent:
\begin{enumerate}
\item[(i)] ${_{s,t}}T_{\mu}$ is compact  from  $b^{p_{1}}_{\alpha_{1}}$ to $b^{p_{2}}_{\alpha_{2}}$.
\item[(ii)] $\kappa$ is a vanishing $(\zeta,\gamma)$-Bergman-Carleson measure.
\end{enumerate}
\end{theorem}

 The \textit{holomorphic} analogues of Theorems \ref{Theorem-1} and \ref{Theorem-2} are proved in \cite{PZ} for $-1<\alpha_{1},\alpha_{2}< \infty$ and for the classical Toeplitz operator ${_{\alpha}}T_{\mu}$ with $\alpha>-1$.

In this paper, we will also provide a criteria for the positive Toeplitz operators ${_{s,t}}T_{\mu}$ on $b^{2}_{\alpha}$ with $\alpha\in \mathbb{R}$ to be in the Schatten classes $S_{p}$ of $b^{2}_{\alpha}$ (see Section \ref{Positive S-C Toeplitz Op}) for $1\leq p <\infty$. The membership in the Schatten classes $S_{p}$, have been studied in various settings; see \cite{CLN1} and \cite{M2} for $1\leq p <\infty$, $\alpha=0$, \cite{M2} for $1\leq p <\infty$, $\alpha>-1$ and \cite{CKL} for $0< p <\infty$, $\alpha=0$. We extend their characterizations for classical positive Toeplitz operators on harmonic  Bergman spaces to ${_{s,t}}T_{\mu}$ on $b^{2}_{\alpha}$ and to all $\alpha\in \mathbb{R}$. To state our result we briefly introduce some notation. Given $\mu\geq 0$, $\hat{\mu}_{\alpha,\delta}$ denotes the weighted averaging function over pseudohyperbolic balls with radius $\delta$, and $\widetilde{\mu}_{\Phi,\alpha,2}$ denotes the $(\Phi,\alpha,2)$-Berezin transform of $\mu$ for $\Phi>-1$ and $\alpha \in \mathbb{R}$. See Section \ref{s-carleson} for relevant definitions. Note that a sequence $\{a_{k}\}$ will always refer to the sequence chosen in Lemma 2.4 below.
The next theorem is the main result.
\begin{theorem}\label{Theorem-3}
Let $1\leq p<\infty$ and $\alpha \in \mathbb{R}$, and $s$ satisfying $2s-\alpha>-1$ be given.
Let $\mu$ be a positive Borel measure on $\mathbb{B}$, and put $u=s-\alpha$ and $\Phi=2s-\alpha$. Then the following statements are equivalent:
\begin{enumerate}
\item[(i)] ${_{s,u}}T_{\mu}:b^{2}_{\alpha}\to b^{2}_{\alpha}$ belongs to $S_{p}$.
\item[(ii)] The $(\Phi,\alpha,2)$-Berezin transform $\widetilde{\mu}_{\Phi,\alpha,2}$ belongs to $L^{p}_{-n}$.
\item[(iii)] The weighted averaging function $\widehat{\mu}_{\alpha,\delta}$ belongs to $L^{p}_{-n}$.
\item[(iv)] The sequence $\{\widehat{\mu}_{\alpha,\delta}(a_{k})\}$ belongs to $\ell^{p}$.
\end{enumerate}
\end{theorem}

The \textit{holomorphic} analogues of our characterizations for \textit{holomorphic} Dirichlet spaces have been obtained   in \cite{AK}; these results are all concerned with Toeplitz operators from a Dirichlet space into itself.

The paper is organized as follows. The notation and some preliminary results
are summarized in Section  \ref{s-preliminaries}.   We will recall various characterizations of (vanishing) $(\lambda,\alpha)$-Bergman-Carleson measures for weighted harmonic Bergman spaces in Section \ref{s-carleson}.
Sections \ref{proof1} and \ref{proof2} are devoted to the proof of our main results, Theorem \ref{Theorem-1} and \ref{Theorem-2}, respectively. Finaly in Section \ref{Positive S-C Toeplitz Op}, we will prove Theorem \ref{Theorem-3}.
Our results attest to the fact that the Toeplitz operators between harmonic Bergman-
Besov spaces  are natural extensions of classical Bergman-space Toeplitz operators on harmonic Bergman spaces.

In the following for two positive expressions $X$ and $Y$ we write $X\lesssim Y$ if there exists a positive constant $C$, whose exact value is inessential, such that $X\leq CY$. If both $X\lesssim Y$ and $Y\lesssim X$, we write $X\sim Y$.
\section{Preliminaries}\label{s-preliminaries}

In this section we collect some known facts that will be used in later sections.

The Pochhammer symbol $(a)_b$ is defined by
\[
(a)_b=\frac{\Gamma(a+b)}{\Gamma(a)},
\]
when $a$ and $a+b$ are off the pole set $-\mathbb{N}$ of the gamma function. By Stirling formula,

\begin{equation}\label{Stirling}
\frac{(a)_c}{(b)_c} \sim c^{a-b} \quad (c\to\infty).
\end{equation}
Define the Rademacher functions $r_k$ on $\mathbb{R}$ by

\begin{align*}
  r_1(\tau) &=
  \begin{cases}
  1, &\text{if $0\leq \tau-\lfloor\tau\rfloor < 1/2$},\\
  \noalign{\medskip}
  -1, &\text{if $1/2 \leq \tau-\lfloor\tau\rfloor <1$}
  \end{cases} \\
  \noalign{\medskip}
  r_k(\tau)&= r_1(2^{k-1} \tau) \quad (k=2,3,\dots).
\end{align*}
Let $\{c_k\}\in \ell^2$ be a sequence of complex numbers and $f(\tau)=\sum_{k=1}^{\infty} c_k r_k(\tau)$. Khinchine's inequality states that for any $0<q<\infty$, the $L^q[0,1]$ norm of $f$ is comparable to the $\ell^2$ norm of $\{c_k\}$.

\begin{lemma}[Khinchine's Inequality]
Let $0<q<\infty$ and $\{c_k\}\in\ell^2$. The series $\sum_{k=1}^\infty c_k r_k(\tau)$ converges almost everywhere and if $f(t)=\sum_{k=1}^\infty c_k r_k(\tau)$, then
\[
\left(\int_0^1 |f(\tau)|^q\, dt\right)^{1/q} \sim \left(\sum_{k=1}^\infty |c_k|^2\right)^{1/2}.
\]
\end{lemma}
A proof of Khinchine's inequality can be found in \cite[\S V.8]{ZY}.

A harmonic function $f$ on $\mathbb{B}$ has a homogeneous expansion, that is, there exist homogeneous harmonic polynomials $f_k$ of degree $k$ such that $f(x)=\sum_{k=0}^\infty f_k(x)$. The series uniformly and absolutely converges on compact subsets of $\mathbb{B}$.

\subsection{Pseudohyperbolic metric}

The canonical M\"obius transformation on $\mathbb{B}$ that exchanges $a$ and $0$ is
\[
\varphi_a(x) = \frac{(1-|a|^2) (a-x) + |a-x|^2 a}{[x,a]^2}.
\]
Here the bracket $[x,a]$ is defined by
\[
[x,a]=\sqrt{1-2 x\cdot a+|x|^2 |a|^2},
\]
where $x\cdot a$ denotes the inner product of $x$ and $a$ in $\mathbb{R}^n.$ The pseudohyperbolic distance between $x,y\in \mathbb{B}$ is
\[
\rho(x,y)=|\varphi_x(y)|=\frac{|x-y|}{[x,y]}.
\]

For a proof of the following lemma see \cite[Lemma 2.2]{CKL}.

\begin{lemma}\label{Inequality}
Let $a,x,y\in \mathbb{B}$. Then

\[
\frac{1-\rho(x,y)}{1+\rho(x,y)} \leq \frac{[x,a]}{[y,a]} \leq \frac{1+\rho(x,y)}{1-\rho(x,y)}.
\]
\end{lemma}

The following two lemmas show that if $x,y\in\mathbb{B}$ are close in the pseudohyperbolic metric, then certain quantities are comparable. Both of them easily follow from Lemma \ref{Inequality} (note that $[x,x]=1-|x|^2$).

\begin{lemma}\label{x-y-close}
Let $0<\delta<1$. Then
\[
[x,y]\sim 1-|x|^2 \sim 1-|y|^2,
\]
for all $x,y\in \mathbb{B}$ with $\rho(x,y)<\delta$.
\end{lemma}

\begin{lemma}\label{Bracket-Hyper}
Let $0<\delta<1$. Then
\[
[x,a] \sim [y,a],
\]
for all $a,x,y\in \mathbb{B}$ with $\rho(x,y)<\delta$.
\end{lemma}

For $0<\delta<1$ and $x\in\mathbb{B}$ we denote the pseudohyperbolic ball with center $x$ and radius $\delta$ by $E_\delta(x)$. The pseudohyperbolic ball $E_\delta(x)$ is also a Euclidean ball with center $c$ and radius $r$, where
\[
c=\frac{(1-\delta^2)x}{1-\delta^2|x|^2} \qquad \text{and} \qquad r=\frac{(1-|x|^2)\delta}{1-\delta^2|x|^2}.
\]
It follows that for fixed $0<\delta<1$, we have $\nu(E_\delta(x))\sim (1-|x|^2)^n$. More generally, for $\alpha\in \mathbb{R}$, by Lemma \ref{x-y-close}
\begin{equation}\label{hyper-volume}
\nu_\alpha(E_\delta(x))=\frac{1}{V_\alpha}  \int_{E_\delta(x)} (1-|y|^2)^\alpha \, d\nu(y) \sim (1-|x|^2)^\alpha \nu(E_\delta(x)) \sim (1-|x|^2)^{\alpha+n}.
\end{equation}

Let $\{a_k\}$ be a sequence of points in $\mathbb{B}$ and $0<\delta<1$. We say that $\{a_k\}$ is $\delta$-separated if $\rho(a_j,a_k)\geq\delta$ for all $j\neq k$. For a proof of the following lemma see, for example, \cite{L2}.
\begin{lemma}\label{ak}
  Let $0<\delta<1$. There exists a sequence of points $\{a_k\}$ in $\mathbb{B}$ satisfying the following properties:
  \begin{enumerate}
    \item[(i)] $\{a_k\}$ is $\delta$-separated.
    \item[(ii)] $\displaystyle \bigcup_{k=1}^\infty E_\delta(a_k) = \mathbb{B}$.
    \item[(iii)] There exists a positive integer $N$ such that every $x\in \mathbb{B}$ belongs to at most $N$ of the balls $E_\delta(a_k)$.
  \end{enumerate}
\end{lemma}

In what follows whenever we use expressions like $\widehat{\mu}_{\alpha,\delta}(a_{k})$, the sequence $\{a_k\}=\{a_k(\delta)\}$ will always refer to the sequence chosen in Lemma
\ref{ak}.

If $u$ is harmonic on a domain $\Omega \subset \mathbb{R}^{n}$, then $|u|^{p}$ is subharmonic on $\Omega$ when $1\leq p<\infty$. This is no longer true when $0<p<1$, nevertheless it is shown in  \cite[Lemma 2]{Fefferman} and \cite{Kuran} that $|u|^{p}$ has subharmonic behaviour in the following sense:
There exists a constant $K \, (\geq1)$ depending only on $n$ and $p$ such that
\begin{equation}\label{equsubharmonic}
 |u(x)|^{p}\leq \frac{K}{r^{n}}\int_{B(x,r)} |u(y)|^{p} d\nu(y),
\end{equation}
whenever $B(x,r)=\{y:|y-x|<r\}\subset \Omega$. In particular, if $p\geq 1$  then the constant $K=1$. Using this result and $E_{\delta}(x)$ is an Euclidean ball one can get the
following useful inequality easily: given $0<p<\infty$, $\alpha>-1$ and $0<\delta<1$, we have
\begin{equation}\label{equsubharmonic2}
 |u(x)|^{p}\lesssim  \frac{1}{(1-|x|^2)^{n+\alpha}}\int_{E_{\delta}(x)} |u(y)|^{p} (1-|y|^2)^{\alpha}d\nu(y)
\end{equation}
for all $u\in b^p_\alpha$ and $x\in \mathbb{B}$.  The  estimate (\ref{equsubharmonic}) also leads to the following pointwise estimate. See
\cite[Lemma 3.1]{DOG} for a proof.
\begin{lemma}\label{growth}
Let $0<p<\infty$ and $\alpha>-1$. Then
\[
|u(x)| \lesssim \frac{\|u\|_{b^p_\alpha}}{(1-|x|^2)^{(n+\alpha)/p}}
\]
for all $u\in b^p_\alpha$ and $x\in \mathbb{B}$.
\end{lemma}

Lastly we  obtain a generalized subharmonicity property with respect to  the measure $\nu_{\alpha}$ on $\mathbb{B}$.
The proofs for $\alpha=0$ given in \cite[Lemma 10]{M2} for the ball  and \cite[Lemma 3.2]{CLN1} for bounded smooth domains work equally well for other $\alpha$ too. A final use of Jensen inequality  extends the result to $p>1$.

\begin{lemma}\label{subharmonic-measure}
Suppose  $1\leq p<\infty$, $\alpha\in \mathbb{R}$, $0<\delta<1$ and $\mu$ is a positive measure on $\mathbb{B}$, Then there exists a constant $C$ (depending only on $\delta$) such that

\begin{equation*}
 \mu(E_{\delta}(x))^{p}\leq \frac{C}{v_{\alpha}(E_{\delta}(x))}\int_{E_{\delta}(x)} \mu(E_{\delta}(y))^{p} d\nu_{\alpha}(y)
\end{equation*}
for all $x$ in $\mathbb{B}$.
\end{lemma}

\subsection{Reproducing Kernels and the Operators $D^t_s$}
For every $\alpha\in \mathbb{R}$ we have $\gamma_{0} (\alpha)=1$, and therefore
\begin{equation}\label{Rq(x,0)}
R_\alpha(x,0)=R_\alpha(0,y)=1, \quad (x,y\in \mathbb{B}, \alpha\in \mathbb{R}).
\end{equation}
Checking the two cases in (\ref{gamma k q-Definition}), we have by (\ref{Stirling})
\begin{equation}\label{gamma-k-asymptotic}
\gamma_k(\alpha) \sim k^{1+\alpha} \quad (k\to \infty).
\end{equation}
$R_\alpha(x,y)$ is harmonic as a function of either of its variables on $\overline{\mathbb{B}}$.
Using the coefficients in the extended kernels we define the radial differential operators $D^t_s$.

\begin{definition}
Let $f=\sum_{k=0}^\infty f_k\in h(\mathbb{B})$ be given by its homogeneous expansion. For $s,t\in\mathbb{R}$ we define  $D_s^t$ on $ h(\mathbb{B}) $  by

\begin{equation*}
  D_s^t f := \sum_{k=0}^\infty \frac{\gamma_k(s+t)}{\gamma_k(s)} \, f_k.
\end{equation*}

\end{definition}
By (\ref{gamma-k-asymptotic}), $\gamma_k(s+t)/\gamma_k(s) \sim k^t$ for any $s,t$ and, roughly speaking, $D_s^t$ multiplies the $k$th homogeneous part of $f$ by $k^{t}$. For every $s\in \mathbb{R}$, $D_s^0=I$, the identity. An important property of $D^t_s$ is that it is invertible with two-sided inverse $D_{s+t}^{-t}$:

\begin{equation}\label{*}
D^{-t}_{s+t} D^t_s = D^t_s D^{-t}_{s+t} = I,
\end{equation}
which follows from the additive property $D_{s+t}^{z} D_s^t = D_s^{z+t}$.

For every $s,t \in \mathbb{R}$, the map $D^t_s: h(\mathbb{B})\to h(\mathbb{B})$ is continuous in the topology
of uniform convergence on compact subsets (see \cite[Theorem 3.2]{GKU2}). The parameter $s$ plays a minor role and is used to have the precise relation

\begin{equation}\label{**}
D_s^t R_s(x,y)=R_{s+t}(x,y)
\end{equation}

The most important property of the operator $D^t_s$ that we will use later is that it allows us to pass from one Bergman-Besov  space to another. More precisely, we have the following isomorphism.

\begin{lemma}\label{Apply-Dst}
Let $0<p<\infty$ and $\alpha,s,t\in \mathbb{R}$.
The map $D^t_s:b^p_\alpha \to b^p_{\alpha+pt}$ is an isomorphism.
\end{lemma}

For a proof of the above lemma see \cite[Corollary 9.2]{GKU2} when $1\leq p<\infty$ and \cite{DOG} when $0<p<1$. We also need the following duality result.

\begin{theorem}\label{Theorem-Dual-of-bpalpha}
Suppose  $\alpha>-1$, $\beta>-1$ and  $1<p<\infty$. Then
the dual of $b^p_\alpha$ can be identified with $b^{p'}_\beta$  under the pairing

\begin{equation}\label{Dual-Pairing}
[f,g ]_{b^{2}_{\gamma}} = \int_{\mathbb{B}} f(x) \, \overline{g(x)} \ d\nu_{\gamma}(x) \qquad (f\in b^p_\alpha, \ g\in b^{p'}_\beta),
\end{equation}
where $1/p+1/p'=1$, $\gamma=\alpha/p+\beta/p'$.
\end{theorem}

\begin{proof}
 If $g\in b^{p'}_\beta$,
 \begin{equation*}
F(f)=[f,g ]_{b^{2}_{\gamma}} =\frac{1}{V_{\gamma}} \int_{\mathbb{B}} (1-|x|^{2})^{\alpha/p}f(x) \, \overline{(1-|x|^{2})^{\beta/p'}g(x)} \ d\nu(x) \quad (f\in b^p_\alpha).
\end{equation*}
 It follows from H\"{o}lder's inequality that $F$ is a bounded linear functional on  $b^p_\alpha$ with  $\|F\|\lesssim \|g\|_{b^{p'}_{\beta}}$ .

Conversely, if $F$ is a bounded linear functional on  $b^p_\alpha$, then according to the Hahn-Banach extension theorem, $F$ can be extended (without increasing its norm) to a bounded linear functional on $L^p_\alpha$. By the usual duality of $L^p$ spaces, there exists some $h \in L^{p'}_{\alpha}$ such that

\begin{equation*}
F(f)=\int_{\mathbb{B}} f(x) \, \overline{h(x)} \ d\nu_{\alpha}(x), \qquad (f\in L^p_\alpha).
\end{equation*}
Let $H(x)=\frac{V_{\gamma}}{V_{\alpha}}(1-|x|^{2})^{(\alpha-\beta)/p'}h(x)$, $x\in \mathbb{B}$. Then  $H \in L^{p'}_{\beta}$ and
\begin{equation*}
F(f)=\int_{\mathbb{B}} f(x) \, \overline{H(x)} \ d\nu_{\gamma}(x) \qquad (f\in b^p_\alpha).
\end{equation*}
It is easy to check that condition $\alpha>-1$ is equivalent to $p'(\gamma+1)>\beta+1$, and the condition $\beta>-1$ is equivalent to $p(\gamma+1)>\alpha+1$. So, by  \cite[Theorem 1.4]{GKU2}, $Q_{\gamma}$ is a bounded projection from $L^p_\alpha$ onto $b^p_\alpha$, and $Q_{\gamma}$ is also a bounded projection from $L^{p'}_{\beta}$ onto $b^{p'}_{\beta}$. Let $g=Q_{\gamma}(H)$. Then $g \in b^{p'}_{\beta}$ and

\begin{equation*}
F(f)=[f,H ]_{b^{2}_{\gamma}}=[f,Q_{\gamma}(H )]_{b^{2}_{\gamma}}=[f,g ]_{b^{2}_{\gamma}}
\end{equation*}
for all $f\in b^p_\alpha$. The proof is now complete.
\end{proof}
For a special case of the preceding theorem when $\alpha=\beta$, see \cite[Corollary 5.1]{LSR}, \cite[Theorem 3.2]{JP} and also \cite[Theorem 13.4]{GKU2} for $\alpha \in \mathbb{R}$ without restriction. In this case, we clearly have $ \gamma=\alpha$ as well.
\subsection{Estimates of Reproducing Kernels}

In case $\alpha>-1$, the reproducing kernels $R_\alpha(x,y)$ are well-studied by various authors. We recall some of their properties below. For extension of these properties to $\alpha\in\mathbb{R}$ we refer to \cite{GKU2}.

For a proof of the following pointwise estimate see \cite{CR,R} when $\alpha>-1$ and \cite{GKU2} when $\alpha\in\mathbb{R}$.

\begin{lemma}\label{R-alpha}
Let $\alpha\in\mathbb{R}$. For all $x,y\in\mathbb{B}$,
\[
|R_\alpha(x,y)| \lesssim
\begin{cases}
\dfrac{1}{[x,y]^{\alpha+n}},&\text{if $\, \alpha>-n$};\\
1+\log \dfrac{1}{[x,y]},&\text{if $\, \alpha=-n$};\\
1,&\text{if $\, \alpha<-n$}.
\end{cases}
\]
\end{lemma}

The next lemma shows that the first part of the above estimate continues to hold when $x$ and $y$ are close enough in the pseudohyperbolic metric. It can be proved along the same lines as \cite[Proposition 5]{M}.

\begin{lemma}\label{Kernel-two-sided}
Let $\alpha>-n$. There exists $0<\delta<1$ such that for every $x\in \mathbb{B}$ and $y\in E_\delta(x)$,
\[
R_\alpha (x,y) \sim \frac{1}{(1-|x|^2)^{\alpha+n}}.
\]
\end{lemma}

The next lemma gives an estimate of weighted $L^p$ norms of reproducing kernels. When $\alpha>-1$ and $c>0$, it is proved in \cite[Proposition 8]{M}. For a full proof see \cite[Theorem 1.5]{GKU2}.

\begin{lemma}\label{norm-kernel}
Let $\alpha\in \mathbb{R}$, $0<p<\infty$ and $\beta>-1$. Set $c=p(\alpha+n)-(\beta+n)$. Then
\[
\int_{\mathbb B}|R_\alpha(x,y)|^p\,(1-|y|^2)^\beta\,d\nu(y)
\sim\begin{cases}
\dfrac1{(1-|x|^2)^c},&\text{if $c>0$};\\
\noalign{\medskip}
1+\log\dfrac1{1-|x|^2},&\text{if $c=0$};\\
\noalign{\medskip}
1,&\text{if $c<0$}.
\end{cases}
\]
\end{lemma}

By Lemma \ref{R-alpha} when $\alpha>-n$, the kernel $R_\alpha(x,y)$ is dominated by $1/[x,y]^{\alpha+n}$. The next lemma estimates the weighted integrals of these dominating terms. For a proof see, for example, \cite[Proposition 2.2]{LS} or \cite[Lemma 4.4]{R}.

\begin{lemma}\label{Integral-[x,y]}
Let $\beta>-1$ and $s\in \mathbb{R}$. Then

\begin{equation*}
  \int_{\mathbb{B}} \frac{(1-|y|^2)^\beta}{[x,y]^{\beta+n+s}} \, d\nu(y) \sim
      \begin{cases}
         \dfrac{1}{(1-|x|^2)^s}, &\text{if $\, s > 0$};\\
         1+\log \dfrac{1}{1-|x|^2}, &\text{if $\, s=0$}; \\
         1, &\text{if $\, s<0$}.
     \end{cases}
\end{equation*}

\end{lemma}

\section{ $(\lambda,\alpha)$-Bergman-Carleson Measures}\label{s-carleson}
Characterizations of $(\lambda,\alpha)$-Bergman-Carleson measures for weighted harmonic Bergman spaces $b^p_\alpha$ ($\alpha>-1$) in terms of $\widehat{\mu}_{\alpha,\delta}$ and $\widetilde{\mu}_{\alpha,t}$ are established by various authors in more general settings. In this subsection we will recall these results.

For $0<\delta<1$ the averaging function $\widehat{\mu}_\delta$ is defined by
\[
\widehat{\mu}_\delta(x)=\frac{\mu(E_\delta(x))}{\nu(E_\delta(x))} \qquad(x\in\mathbb{B}).
\]
More generally, for $\alpha \in \mathbb{R}$ we define
\[
\widehat{\mu}_{\alpha,\delta}(x):=\frac{\mu(E_\delta(x))}{\nu_\alpha(E_\delta(x))} \qquad(x\in\mathbb{B}).
\]
By (\ref{hyper-volume}), $\widehat{\mu}_{\alpha,\delta}(x)\sim \mu(E_\delta(x))/(1-|x|^2)^{\alpha+n}$. The following lemma shows that weighted $L^p$ behaviour of $\widehat{\mu}_{\alpha,\delta}$ is independent of $\delta$. For a proof see \cite[Proposition 3.6]{CKL} when $\alpha=0$. The proof also works for all $\alpha\in\mathbb{R}$.

\begin{lemma}\label{Lp-mu-delta}
Let $0<p<\infty$, $\alpha, \beta \in \mathbb{R}$ and $0<\delta, \epsilon<1$. Then $\widehat{\mu}_{\alpha,\delta} \in L^p_\beta$ if and only if $\widehat{\mu}_{\alpha,\epsilon} \in L^p_\beta$.
\end{lemma}

For $\Phi>-1$ and $\alpha \in \mathbb{R}$, the $(\Phi,\alpha,2)$-Berezin transform of a positive measure $\mu$ on $\mathbb{B}$ is
\[
\widetilde{\mu}_{\Phi,\alpha,2}(x)=\int_\mathbb{B} \frac{|R_{\Phi}(x,y)|^2}{\|R_{\Phi}(x,\cdot)\|^2_{L^2_{\Phi}}}(1-|y|^2)^{\Phi-\alpha} \, d\mu(y).
\]

When $\Phi=\alpha>-1$ and $t>1$, the $(\alpha,t)$-Berezin transform of $\mu$ is defined by
\[
\widetilde{\mu}_{\alpha,t}(x):=\int_\mathbb{B} \frac{|R_\alpha(x,y)|^t}{\|R_\alpha(x,\cdot)\|^t_{L^t_\alpha}} \, d\mu(y).
\]
Since $(\alpha+n)t-(\alpha+n)>0$, by Lemma \ref{norm-kernel}
\begin{equation}\label{Kernel-Norm-Sim}
\widetilde{\mu}_{\alpha,t}(x)  \sim (1-|x|^2)^{(\alpha+n)t-(\alpha+n)} \int_\mathbb{B} |R_\alpha(x,y)|^t \, d\mu(y).
\end{equation}
Applying also Lemma \ref{R-alpha} we obtain the following estimate
\begin{equation}\label{Berezin-type}
\widetilde{\mu}_{\alpha,t}(x) \lesssim (1-|x|^2)^{(\alpha+n)t-(\alpha+n)} \int_\mathbb{B} \frac{d\mu(y)}{[x,y]^{(\alpha+n)t}}.
\end{equation}
Using the dominating term on the right-hand side, for $\alpha>-1$ and $s>0$, we define $(\alpha,s)$-Berezin-type transform $\bar{\mu}_{\alpha,s}$ by
\[
\bar{\mu}_{\alpha,s}(x):=(1-|x|^2)^s \int_\mathbb{B} \frac{d\mu(y)}{[x,y]^{\alpha+n+s}}.
\]

The following proposition shows $L^p_\alpha$ behaviour of $\widetilde{\mu}_{\alpha,t}$, $\bar{\mu}_{\alpha,s}$ and $\widehat{\mu}_{\alpha,\delta}$ are same when $p>1$. For a  proof see  \cite[Proposition 3.2]{DU2}.

\begin{proposition}\label{Lp-Berezin}
  Let $1<p<\infty$ and $\alpha>-1$. The following are equivalent:
\begin{enumerate}
  \item[(a)] $\widehat{\mu}_{\alpha,\delta} \in L^p_\alpha$ for some (every) $0<\delta<1$.
  \item[(b)] $\bar{\mu}_{\alpha,s} \in L^p_\alpha$ for some (every) $s>0$.
  \item[(c)] $\widetilde{\mu}_{\alpha,t} \in L^p_\alpha$ for some (every) $t>1$.
\end{enumerate}
\end{proposition}

The next proposition is about a similar result concerning pointwise bounds. For a  proof see  \cite[Proposition 3.3]{DU2}.

\begin{proposition}\label{Pointwise}
Suppose $\gamma \geq 0$ and $\alpha>-1$. The following are equivalent:
\begin{enumerate}
  \item[(a)] $\widehat{\mu}_{\alpha,\delta}(x) \lesssim (1-|x|^2)^\gamma$ for some (every) $0<\delta<1$.
  \item[(b)] $\bar{\mu}_{\alpha,s}(x) \lesssim (1-|x|^2)^\gamma$ for some (every) $s>\gamma$.
  \item[(c)] $\widetilde{\mu}_{\alpha,t}(x) \lesssim (1-|x|^2)^\gamma$ for some (every) $t>(\alpha+n+\gamma)/(\alpha+n)$.
\end{enumerate}
\end{proposition}

The characterizations of $(\lambda,\alpha)$-Bergman-Carleson measures divided into two cases depending on whether $q<p$ or $q\geq p$. In the case $q<p$ note that the conjugate exponent of $1/\lambda=p/q$ is $1/(1-\lambda)=p/(p-q)$.
\begin{theorem}\label{Carleson-Besov1}
  Let $0<q<p<\infty$, $\alpha>-1$ and $\mu\geq0$. The following are equivalent:
  \begin{enumerate}
    \item[(a)] $\mu$ is a $(\lambda,\alpha)$-Bergman-Carleson measure.
    \item[(b)] $\widehat{\mu}_{\alpha,\varepsilon} \in L^{p/(p-q)}_\alpha$ for some (every) $0<\varepsilon<1$.
    \item[(c)] $\widetilde{\mu}_{\alpha,t} \in L^{p/(p-q)}_\alpha$ for some (every) $t>1$.
    \item[(d)] $\bar{\mu}_{\alpha,s} \in L^{p/(p-q)}_\alpha$ for some (every) $s>0$.
    \item[(e)] $\left\{\widehat{\mu}_{\alpha,\delta}(a_{k})(1-|a_{k}|^{2})^{(n+\alpha)(1-q/p)}\right\} \in \ell^{p/(p-q)}$ for some (every) $0<\delta<1$.
  \end{enumerate}
\end{theorem}

\begin{proof}
  That (a) and (b) are equivalent is proved in \cite{L1} and \cite{L2} for the unweighted holomorphic Bergman space on the unit disc $\mathbb{D}$. Note also that the equivalence of the discrete form (e) to (a) and (b) is actually proved therein; see, for example, the proof of \cite[Theorem 1]{L2}. As is mentioned in the remarks of \cite{L1} the method works also for weighted harmonic Bergman spaces on the unit ball of $\mathbb{R}^n$. The equivalence of (a), (b), (c) and (e) for $\alpha=0$ is proved in \cite[Theorem 3.4]{CLN2} not just for the ball but for bounded smooth domains. The proof works equally well for other $\alpha$ too. The equivalence of (b), (c) and (d) follows from Proposition \ref{Lp-Berezin}.
\end{proof}
As a consequence of Theorem \ref{Carleson-Besov1}, for $0 < \lambda<1$, a positive Borel measure $\mu$ on
$\mathbb{B}$ is a $(\lambda,\alpha)$-Bergman-Carleson measure if and only if
\[
\mu(E_\delta(x))(1-|x|^2)^{-n-\alpha} \in L^{1/(1-\lambda)}_\alpha
\]
or
 \[
\left\{\mu(E_\delta(a_{k}))(1-|a_{k}|^{2})^{-(n+\alpha)\lambda}\right\} \in \ell^{1/(1-\lambda)}
\]
for some (every) $0<\delta<1$.

We now consider the case $q\geq p$.

\begin{theorem}\label{Carleson-Besov2}
  Let $0<p\leq q<\infty$, $\alpha>-1$ and $\mu\geq 0$. The following are equivalent:
  \begin{enumerate}
    \item[(a)] $\mu$ is a  $(\lambda,\alpha)$-Bergman-Carleson measure.
    \item[(b)] $\widehat{\mu}_{\alpha,\delta} \lesssim (1-|x|^2)^{(\alpha+n)(q/p-1)}$ for some (every) $0<\delta<1$.
    \item[(c)] $\widetilde{\mu}_{\alpha,t} \lesssim (1-|x|^2)^{(\alpha+n)(q/p-1)}$ for some (every) $t>q/p$.
    \item[(d)] $\bar{\mu}_{\alpha,s} \lesssim (1-|x|^2)^{(\alpha+n)(q/p-1)}$ for some (every) $s>(\alpha+n)(q/p-1)$.
  \end{enumerate}
\end{theorem}

Note that (b) is equivalent to
\[
\mu(E_\delta(x)) \lesssim (1-|x|^2)^{(\alpha+n)q/p} \quad \text{for some (every) $0<\delta<1$}
\]
and (d) is equivalent to
\[
(1-|x|^2)^c \int_\mathbb{B} \frac{d\mu(y)}{[x,y]^{(\alpha+n)q/p+c}} \lesssim 1 \quad \text{for some (every) $c>0$.}
\]

\begin{proof}
   Equivalence of (a), (b) and (c) for $\alpha=0$ is proved in \cite[Theorem 3.1]{CLN2} for bounded smooth domains. The proof works equally well for other $\alpha$ too. That (b), (c) and (d) are equivalent follows from Proposition \ref{Pointwise}.
\end{proof}
By  Theorems \ref{Carleson-Besov1} and \ref{Carleson-Besov2}, the notion of $(\lambda,\alpha)$-Bergman-Carleson measures depend only on $\alpha$ and the ratio $\lambda = q/p$. We also need the following proposition.
\begin{proposition}\label{product-carleson}
Let $\mu$ be a positive Borel measure on $\mathbb{B}$. Let $0<p_{1},p_{2}<\infty$ and $-1<\alpha_{1},\alpha_{2} <\infty$ and
let
\[
\theta=\frac{1}{p_{1}}+\frac{1}{p_{2}}, \qquad \varrho=\frac{1}{\theta}\left(\frac{\alpha_{1}}{p_{1}}+\frac{\alpha_{2}}{p_{2}}\right).
\]
If  $\mu$ is a $(\theta,\varrho)$-Bergman-Carleson measure, then
\begin{equation*}
\int_{\mathbb{B}}|f(x)||g(x)| \, d\mu(x)\lesssim \|f\|_{b^{p_{1}}_{\alpha_{1}}}\|g\|_{b^{p_{2}}_{\alpha_{2}}} \qquad (f\in b^{p_{1}}_{\alpha_{1}}, g \in b^{p_{2}}_{\alpha_{2}}).
\end{equation*}
\end{proposition}
\begin{proof}
Let $f\in b^{p_{1}}_{\alpha_{1}}, g \in b^{p_{2}}_{\alpha_{2}}$. Since $\theta p_{1}>1, \theta p_{2}>1$ and $1/\theta p_{1}+1/\theta p_{2}=1$, we can apply H\"{o}lder's inequality to obtain
\begin{align}
\|&fg\|_{L^{1/\theta}_{\varrho}}\nonumber\\
&=\left(\frac{1}{V_{\varrho}}\int_{\mathbb{B}} |f(x)g(x)|^{1/\theta}(1-|x|^{2})^{\varrho}d\nu(x) \right)^{\theta}\nonumber\\
&\lesssim \left(\int_{\mathbb{B}} |f(x)|^{p_{1}}(1-|x|^{2})^{\alpha_{1}}d\nu(x) \right)^{1/p_{1}}\left(\int_{\mathbb{B}} |g(x)|^{p_{2}}(1-|x|^{2})^{\alpha_{2}}d\nu(x) \right)^{1/p_{2}}\nonumber\\
&\lesssim \|f\|_{b^{p_{1}}_{\alpha_{1}}}\|g\|_{b^{p_{2}}_{\alpha_{2}}}. \label{ineq1}
\end{align}
Thus, $fg\in L^{1/\theta}_{\varrho}$. Let $0<\delta<1$. Because $E_{\delta/2}(x)$ is an Euclidean ball with center $c=(1-(\delta/2)^2)x/(1-(\delta/2)^2|x|^2)$ and the radius behaves like $1-|x|^{2}$ when $\delta/2$ is fixed, it follows from \cite[Lemma 3.3]{DOG} that
\begin{equation*}
 |f(c)g(c)|\lesssim  \frac{1}{(1-|x|^2)^{(n+\varrho)\theta}}\left(\int_{E_{\delta/2}(x)} |f(y)g(y)|^{1/\theta} (1-|y|^2)^{\varrho}d\nu(y)\right)^{\theta}
\end{equation*}
for all  $x\in \mathbb{B}$. Since $\delta/2$ is fixed, the distance from $x$ to the centre of $E_{\delta/2}(x)$ is at most $(\delta/2)|x|$ times the radius of $E_{\delta/2}(x)$. By \cite[Lemma 3.3]{DOG}  again, we  get
\begin{equation*}
 |f(x)g(x)|\lesssim  \frac{1}{(1-|x|^2)^{(n+\varrho)\theta}}\left(\int_{E_{\delta/2}(x)} |f(y)g(y)|^{1/\theta} (1-|y|^2)^{\varrho}d\nu(y)\right)^{\theta}
\end{equation*}
for all  $x\in \mathbb{B}$. For $a\in \mathbb{B}$ and $x \in E_{\delta/2}(a)$, we note that $E_{\delta/2}(x) \subset E_{\delta}(a)$. Let $E_{\delta/2}(a_{k})$ be the associated sets to the sequence $\{a_k\}=\{a_k(\delta/2)\}$ in  Lemma \ref{ak}. Thus we have
\begin{align*}
 &|f(x)g(x)|\\
 & \lesssim  \frac{1}{(1-|x|^2)^{(n+\varrho)\theta}}\left(\int_{E_{\delta/2}(x)} |f(y)g(y)|^{1/\theta} (1-|y|^2)^{\varrho}d\nu(y)\right)^{\theta}\\
 &\lesssim  \frac{1}{(1-|x|^2)^{(n+\varrho)\theta}}\left(\int_{E_{\delta}(a_{k})} |f(y)g(y)|^{1/\theta} (1-|y|^2)^{\varrho}d\nu(y)\right)^{\theta}, \quad x\in E_{\delta/2}(a_{k})
\end{align*}
 for $k=1,2, \dots$. Then by Lemma \ref{ak} and Lemma \ref{x-y-close}, we have
\begin{align}
&\int_{\mathbb{B}} |f(x)g(x)| \,  d\mu(x)\nonumber\\
&\lesssim \sum_{k=1}^{\infty}\int_{E_{\delta/2}(a_{k})} |f(x)g(x)| \, d\mu(x) \nonumber\\
&\lesssim \sum_{k=1}^{\infty}\left(\int_{E_{\delta}(a_{k})} |f(y)g(y)|^{1/\theta} (1-|y|^2)^{\varrho}d\nu(y)\right)^{\theta}\nonumber\\
&\times\left( \int_{E_{\delta/2}(a_{k})}\frac{d\mu(x)}{(1-|x|^2)^{(n+\varrho)\theta}}\right)\nonumber\\
&\lesssim \sum_{k=1}^{\infty} \frac{\mu(E_{\delta/2}(a_{k}))}{(1-|a_{k}|^2)^{(n+\varrho)\theta}}\left(\int_{E_{\delta}(a_{k})} |f(y)g(y)|^{1/\theta} (1-|y|^2)^{\varrho}d\nu(y)\right)^{\theta}.\label{ineq2}
\end{align}
First, assume that  $\theta\geq 1$. Since $\mu$ is a $(\theta,\varrho)$-Bergman-Carleson measure, by Theorem \ref{Carleson-Besov2} we have
\begin{equation*}
 \mu(E_{\delta/2}(a_{k}))\lesssim (1-|a_{k}|^{2})^{(n+\varrho)\theta}.
\end{equation*}
Then it follows from this together with (\ref{ineq2}) and Lemma \ref{ak} that
\begin{align}
\int_{\mathbb{B}}|f(x)g(x)| \, d\mu(x)&\lesssim \sum_{k=1}^{\infty} \left(\int_{E_{\delta}(a_{k})} |f(y)g(y)|^{1/\theta} (1-|y|^2)^{\varrho}d\nu(y)\right)^{\theta}\nonumber \\
&\lesssim  \left(\sum_{k=1}^{\infty} \int_{E_{\delta}(a_{k})} |f(y)g(y)|^{1/\theta} (1-|y|^2)^{\varrho}d\nu(y)\right)^{\theta}\nonumber \\
&\leq L \|fg\|_{L^{1/\theta}_{\varrho}}\lesssim \|fg\|_{L^{1/\theta}_{\varrho}} \label{ineq3},
\end{align}
where $L$ is the number provided by  \cite[Lemma 3]{L2}.
Next assume that $0<\theta<1$.
Then by using H\"{o}lder's inequality in (\ref{ineq2}), Lemma \ref{ak} and Lemma \ref{x-y-close}, we get
\begin{align*}
&\int_{\mathbb{B}}|f(x)g(x)| \, d\mu(x)\\
&\lesssim \sum_{k=1}^{\infty} \frac{\mu(E_{\delta/2}(a_{k}))}{(1-|a_{k}|^2)^{(n+\varrho)\theta}}\left(\int_{E_{\delta}(a_{k})} |f(y)g(y)|^{1/\theta} (1-|y|^2)^{\varrho}d\nu(y)\right)^{\theta}\nonumber \\
 &\lesssim\left\{ \sum_{k=1}^{\infty}\left[\frac{\mu(E_{\delta/2}(a_{k}))}{(1-|a_{k}|^2)^{(n+\varrho)\theta }}\right]^{1/(1-\theta)}\right\}^{(1-\theta)} \nonumber \\
 &\times\left(\sum_{k=1}^{\infty}\int_{E_{\delta}(a_{k})} |f(y)g(y)|^{1/\theta} (1-|y|^2)^{\varrho}d\nu(y)\right)^{\theta}.
\end{align*}
Since  $\mu$ is a $(\theta,\varrho)$-Bergman-Carleson measure, by Theorem \ref{Carleson-Besov1} we obtain (\ref{ineq3}) again. Combining (\ref{ineq1}) and (\ref{ineq3}) concludes the proof.
\end{proof}

\subsection{ Vanishing $(\lambda,\alpha)$-Bergman-Carleson Measures}\label{ss-Vanishing-BC}

In this subsection we will characterize vanishing $(\lambda,\alpha)$-Bergman-Carleson measures. The next proposition is about a  result concerning pointwise bounds.

\begin{proposition}\label{Pointwise2}
Suppose $\gamma \geq 0$ and $\alpha>-1$. The following are equivalent:
\begin{enumerate}
  \item[(a)] $\lim_{|x|\to 1^{-}}(1-|x|^2)^{-\gamma} \, \widehat{\mu}_{\alpha,\delta}(x) =0 $ for some (every) $0<\delta<1$.
  \item[(b)] $\lim_{|x|\to 1^{-}}(1-|x|^2)^{-\gamma} \, \bar{\mu}_{\alpha,s}(x) =0$ for some (every) $s>\gamma$.
  \item[(c)] $\lim_{|x|\to 1^{-}}(1-|x|^2)^{-\gamma} \, \widetilde{\mu}_{\alpha,t}(x)=0$ for some (every) $t>(\alpha+n+\gamma)/(\alpha+n)$.
\end{enumerate}
\end{proposition}

\begin{proof} The proof is similar to the proof of \cite[Proposition 3.3]{DU2}.
To see that (a) implies (b) suppose that (a) holds for some $0<\delta<1$. By \cite[Eq. (3.3)]{DU2}

\begin{align*}
(1-|x|^2)^{-\gamma} \, \bar{\mu}_{\alpha,s}(x) & \sim (1-|x|^2)^{s-\gamma} \int_\mathbb{B} \frac{(1-|y|^2)^{-\gamma}\widehat{\mu}_{\alpha,\delta}(y)}{[x,y]^{\alpha+n+s}} \, d\nu_{\alpha+\gamma}(y).
\end{align*}
Since $(1-|y|^2)^{-\gamma}\widehat{\mu}_{\alpha,\delta}(y)$ is continuous on $\mathbb{B}$ and $\lim_{|x|\to 1^{-}}(1-|x|^2)^{-\gamma} \, \widehat{\mu}_{\alpha,\delta}(x) \\ =0 $, by \cite[ Remark 3.3]{DU1}
part (b) holds for every $s>\gamma$. That (b) implies (c) is immediate from (\ref{Berezin-type}). To see that (c) implies (a), pick $\delta_0$ as in Lemma \ref{Kernel-two-sided}.  \cite[Eq. (3.4)]{DU2} shows that (a) holds with $\delta=\delta_0$. That it holds for every $0<\delta<1$ is a consequence of Lemma 3.2 of \cite{CLN1} and Lemma \ref{x-y-close}.
\end{proof}
 Characterizations of vanishing $(\lambda,\alpha)$-Bergman-Carleson measures are also divided into two cases as whether $\lambda\geq 1$ or $0<\lambda<1$. We first consider the case  $\lambda\geq 1$.

\begin{theorem}\label{Carleson-Besov3}
  Let $0<p\leq q<\infty$, $\lambda=q/p$,  $\alpha>-1$ and $\mu\geq 0$. The following are equivalent:
  \begin{enumerate}
    \item[(a)] $\mu$ is a vanishing $(\lambda,\alpha)$-Bergman-Carleson measure.
    \item[(b)] $\lim_{|x|\to 1^{-}}(1-|x|^2)^{(\alpha+n)(1-\lambda)} \, \widehat{\mu}_{\alpha,\varepsilon}(x) =0 $ for some (every) $0<\varepsilon<1$.
     \item[(c)] $\lim_{k\to \infty}(1-|a_{k}|^{2})^{(n+\alpha)(1-\lambda)}\widehat{\mu}_{\alpha,\delta}(a_{k}) =0$ for some (every) $0<\delta<1$.
    \item[(d)] $\lim_{|x|\to 1^{-}}(1-|x|^2)^{(n+\alpha)(1-\lambda)}\widetilde{\mu}_{\alpha,t}(x)=0 $ for some (every) $t>\lambda$.
    \item[(e)] $\lim_{|x|\to 1^{-}}(1-|x|^2)^{(n+\alpha)(1-\lambda)}\bar{\mu}_{\alpha,s}(x) =0$ for some (every) $s>(\alpha+n)(\lambda-1)$.
  \end{enumerate}

\end{theorem}

Note that (e) is equivalent to
\[
\lim_{|x|\to 1^{-}}(1-|x|^2)^{c}\int_{\mathbb{B}}\frac{d\mu(y)}{[x,y]^{(n+\alpha)\lambda+c}} =0 \quad \text{for some (every) $c>0$}.
\]

\begin{proof}
  Equivalence of (a), (b) and (c) and (d) for $\alpha=0$  is proved in \cite[Theorem 3.5]{CLN2} for bounded smooth domains.  The proof works equally well for other $\alpha$ too. That (b), (d) and (e) are equivalent follows from Proposition \ref{Pointwise2}.
\end{proof}
We now consider the case $0<\lambda<1$.

\begin{theorem}\label{Carleson-Besov4}
  Let $0<q<p<\infty$, $\lambda=q/p$, $\alpha>-1$ and $\mu\geq 0$. The following are equivalent:

  \begin{enumerate}
    \item[(a)] $\mu$ is a $(\lambda,\alpha)$-Bergman-Carleson measure.
    \item[(b)] $\mu$ is a vanishing $(\lambda,\alpha)$-Bergman-Carleson measure.
  \end{enumerate}

\end{theorem}

For a proof of the above theorem see \cite{O} and \cite[Theorem 3.6]{CLN2} for bounded smooth domains.

\section{Proof of Theorem \ref{Theorem-1} } \label{proof1}
In this section we will prove Theorem \ref{Theorem-1}. Before that we present a very useful intertwining relation for transforming certain
problems for Toeplitz operators between harmonic Bergman-Besov spaces to similar problems for classical
Toeplitz operators between weighted harmonic Bergman spaces. The holomorphic version is in \cite{AK}.
\begin{theorem}\label{intertwining relation}
We have $D^t_s ({_{s,t}}T_{\mu})= ({_{s+t}}T_{\kappa}) D^t_s $, where
\[
{_{s+t}}T_{\kappa}f(x)=\frac{V_{\alpha_{1}}}{V_{s}} \int_{\mathbb{B}} R_{s+t}(x,y) f(y) d\kappa(y)
\]
 is classical Toeplitz operator from  $b^{p_{1}}_{\alpha_{1}+p_{1}t}$ to $b^{p_{2}}_{\alpha_{2}+p_{2}t}$. Consequently,
 \[
({_{s,t}}T_{\mu})=D^{-t}_{s +t}({_{s+t}}T_{\kappa}) D^t_{s}, \qquad ({_{s+t}}T_{\kappa})=D^{t}_{s}({_{s,t}}T_{\mu})D^{-t}_{s+t}.
\]
\end{theorem}
\begin{proof}
By differentiation under the integral sign and (\ref{**}), we have
\begin{align*}
D^t_s ({_{s,t}}T_{\mu}f)(x)&=\frac{V_{\alpha_{1}}}{V_{s}} \int_{\mathbb{B}} R_{s+t}(x,y) D^t_{s}f(y) d\kappa(y)\\
&=({_{s+t}}T_{\kappa}) (D^t_sf)(x) \qquad (f\in b^{p_{1}}_{\alpha_{1}}).
\end{align*}
For the second and third assertions, we note that $(D^t_{s})^{-1}=D^{-t}_{s +t}$ by (\ref{*}).
\end{proof}

By Theorem \ref{intertwining relation}, ${_{s,t}}T_{\mu}$ is bounded from $b^{p_{1}}_{\alpha_{1}}$ to $b^{p_{2}}_{\alpha_{2}}$ if and only if  ${_{s+t}}T_{\kappa}$ is bounded from $b^{p_{1}}_{\alpha_{1}+p_{1}t}$ to $b^{p_{2}}_{\alpha_{2}+p_{2}t}$. With all the preparation done in earlier sections, now we are ready to prove the main result.

\subsection{ (i) Implies (ii).}\label{ss-implies1}
We divide this part into two cases: $\zeta\geq1$ and $0<\zeta<1$.

\textbf{Case $1$:} $\zeta\geq1$. Notice that this condition is equivalent to $0<p_{1}\leq p_{2}<\infty$. Fix $x\in \mathbb{B}$ and consider $R_{s+t}(x,.)$. Then under the condition $p_{1}(n+s+t)>n+\alpha_{1}+p_{1}t$ provided by (\ref{Equ-1}), it is easy to check using Lemma \ref{norm-kernel} that $R_{s+t}(x,.) \in b^{p_{1}}_{\alpha_{1}+p_{1}t}$ with
\[
\|R_{s+t}(x,.)\|^{p_{1}}_{\alpha_{1}+p_{1}t}\lesssim (1-|x|^{2})^{(n+\alpha_{1})-p_{1}(n+s)}.
\]
Take $\delta=\delta_{0}$ where $\delta_{0}$ is the number provided by Lemma \ref{Kernel-two-sided}. By Lemma \ref{x-y-close} and Lemma \ref{Kernel-two-sided}, we have
\begin{align*}
\mu(E_{\delta}(x))&\lesssim \frac{V_{\alpha_{1}}}{V_{s}} (1-|x|^{2})^{2n+s+t+\alpha_{1}}\int_{E_{\delta}(x)} |R_{s+t}(x,y)|^{2} (1-|y|^{2})^{s+t-\alpha_{1}} d\mu(y)\\
&\lesssim \frac{V_{\alpha_{1}}}{V_{s}} (1-|x|^{2})^{2n+s+t+\alpha_{1}}\int_{\mathbb{B}} |R_{s+t}(x,y)|^{2} (1-|y|^{2})^{s+t-\alpha_{1}} d\mu(y)\\
&=(1-|x|^{2})^{2n+s+t+\alpha_{1}}{_{s+t}}T_{\kappa}[R_{s+t}(x,.)](x),
\end{align*}
and therefore
\begin{align*}
\widehat{\kappa}_{\gamma,\delta}(x)&=\frac{(1-|x|^{2})^{s+t-\alpha_{1}}\mu(E_\delta(x))}{\nu_\gamma(E_\delta(x))}\\
&\lesssim (1-|x|^{2})^{2(n+s+t)-(n+\gamma)}{_{s+t}}T_{\kappa}[R_{s+t}(x,.)](x).
\end{align*}

On the other hand, by Lemma \ref{growth} together with the boundedness of the Toeplitz operator ${_{s+t}}T_{\kappa}$ and an inequality above, we get

\begin{align*}
{_{s+t}}T_{\kappa}[R_{s+t}(x,.)](x)&= |{_{s+t}}T_{\kappa}[R_{s+t}(x,.)](x)|\\
&\lesssim (1-|x|^{2})^{-t-\frac{n+\alpha_{2}}{p_{2}}}\|{_{s+t}}T_{\kappa}[R_{s+t}(x,.)]\|_{b^{p_{2}}_{\alpha_{2}+p_{2}t}}\\
&\lesssim (1-|x|^{2})^{-t-\frac{n+\alpha_{2}}{p_{2}}}\|{_{s+t}}T_{\kappa}\|\|R_{s+t}(x,.)\|_{b^{p_{1}}_{\alpha_{1}+p_{1}t}}\\
&\lesssim (1-|x|^{2})^{\frac{n+\alpha_{1}}{p_{1}}-\frac{n+\alpha_{2}}{p_{2}}-(n+s+t)}\|{_{s+t}}T_{\kappa}\|,
\end{align*}
where $\|{_{s+t}}T_{\kappa}\|$ denote the operator norm of ${_{s+t}}T_{\kappa}:b^{p_{1}}_{\alpha_{1}+p_{1}t}\to b^{p_{2}}_{\alpha_{2}+p_{2}t}$.
Combining these estimates we have
\begin{align*}
\widehat{\kappa}_{\gamma,\delta}(x)&\lesssim (1-|x|^{2})^{s+t+\frac{n+\alpha_{1}}{p_{1}}-\frac{n+\alpha_{2}}{p_{2}}-\gamma}\|{_{s+t}}T_{\kappa}\|\\
&= (1-|x|^{2})^{\zeta\gamma-\gamma+n(\frac{1}{p_{1}}-\frac{1}{p_{2}})}\|{_{s+t}}T_{\kappa}\|=(1-|x|^{2})^{(n+\gamma)(\zeta-1)}
\|{_{s+t}}T_{\kappa}\|.
\end{align*}
By Theorem \ref{Carleson-Besov2} this means that $\kappa$ is a $(\zeta,\gamma)$-Bergman-Carleson measure.

\textbf{Case $2$:} $0<\zeta<1$. Notice that this condition is equivalent to $0<p_{2}< p_{1}<\infty$. Let $ r_k(\tau)$ be a sequence of Rademacher functions and $\{a_{k}\}$ be any sequence satisfying the conditions of the Lemma \ref{ak}. Since
\begin{equation*}
n+s+1>n\max\left(1,\frac{1}{p_{1}}\right)+\frac{1+\alpha_{1}}{p_{1}},
\end{equation*}
we know from  \cite[Theorem 10.1]{GKU2} when $1\leq p < \infty$ and  \cite[Theorem 1.4]{DOG} when $0< p < 1$ that, for any sequence of real numbers $\{\lambda_{m}\}\in \ell^{p_{1}}$, the function
\begin{equation*}
f_{\tau}(x)=\sum_{k=1}^{\infty} \lambda_{k}r_k(\tau)(1-|a_k|^{2})^{n+s-(n+\alpha_{1})/p_{1}} R_{s+t}(x,a_k)
\end{equation*}
is in $b^{p_{1}}_{\alpha_{1}+p_{1}t}$ with $ \|f_{\tau}||_{b^{p_{1}}_{\alpha_{1}+p_{1}t}}\|\leq \|{\lambda_{k}}\|_{\ell^{p_{1}}} $
for almost every $\tau \in (0,1)$. Let
\begin{equation*}
f_{k}(x)=(1-|a_k|^{2})^{n+s-(n+\alpha_{1})/p_{1}} R_{s+t}(x,a_k).
\end{equation*}
Since ${_{s+t}}T_{\kappa}$ is bounded from $b^{p_{1}}_{\alpha_{1}+p_{1}t}$ to $b^{p_{2}}_{\alpha_{2}+p_{2}t}$, we get that for almost every $\tau \in (0,1)$
\begin{align*}
\|{_{s+t}}T_{\kappa}f_{\tau}\|^{p_{2}}_{b^{p_{2}}_{\alpha_{2}+p_{2}t}}&=\int_{\mathbb{B}} \left|\sum_{k=1}^{\infty} \lambda_{k}r_k(\tau) {_{s+t}}T_{\kappa}f_{k} (x)\right|^{p_{2}}d\nu_{\alpha_{2}+p_{2}t}(x)\\
&\lesssim \|{_{s+t}}T_{\kappa}\|^{p_{2}}\cdot \|f_{\tau}\|^{p_{2}}_{b^{p_{1}}_{\alpha_{1}+p_{1}t}}\lesssim \|{_{s+t}}T_{\kappa}\|^{p_{2}}\cdot \|\lambda_{k}\|^{p_{2}}_{\ell^{p_{1}}}.
\end{align*}
Integrating both sides with respect to $\tau$ from $0$ to $1$, to obtain
\begin{equation*}
\int_{0}^{1}\int_{\mathbb{B}} \left|\sum_{k=1}^{\infty} \lambda_{k}r_k(\tau) {_{s+t}}T_{\kappa}f_{k} (x)\right|^{p_{2}}d\nu_{\alpha_{2}+p_{2}t}(x)d\tau \lesssim \|{_{s+t}}T_{\kappa}\|^{p_{2}}\cdot \|\lambda_{k}\|^{p_{2}}_{\ell^{p_{1}}}.
\end{equation*}
Applying Fubini's theorem shows
\begin{equation*}
\int_{\mathbb{B}}\int_{0}^{1} \left|\sum_{k=1}^{\infty} \lambda_{k}r_k(\tau) {_{s+t}}T_{\kappa}f_{k} (x)\right|^{p_{2}}d\tau d\nu_{\alpha_{2}+p_{2}t}(x) \lesssim \|{_{s+t}}T_{\kappa}\|^{p_{2}}\cdot \|\lambda_{k}\|^{p_{2}}_{\ell^{p_{1}}}.
\end{equation*}
 To use Khinchine's inequality we first check that  $\left\{\lambda_{k} ({_{s+t}}T_{\kappa})f_{k}(x)\right\}$ is in $\ell^{2}$. Replacing $\lambda_{k} $ with $\lambda_{k}r_k(\tau)$ and using ${_{s+t}}T_{\kappa}$ is bounded from $b^{p_{1}}_{\alpha_{1}+p_{1}t}$ to $b^{p_{2}}_{\alpha_{2}+p_{2}t}$, we get that for almost every $\tau \in (0,1)$
\begin{align*}
\left(\sum_{k=1}^{\infty}\left|\lambda_{k}r_k(\tau){_{s+t}}T_{\kappa}f_{k} (x)\right|^{2}\right)^{\frac{1}{2}}&\lesssim \|{_{s+t}}T_{\kappa}\|\cdot \|f_{\tau}\|_{b^{p_{1}}_{\alpha_{1}+p_{1}t}}\\
&\lesssim \|{_{s+t}}T_{\kappa}\|\cdot \|\lambda_{k}\|_{\ell^{p_{1}}}.
\end{align*}
We now apply Khinchine's inequality and deduce
\begin{equation}\label{usingKhin}
\int_{\mathbb{B}}\left(\sum_{k=1}^{\infty}\left|\lambda_{k}\right|^{2} \left|{_{s+t}}T_{\kappa}f_{k} (x)\right|^{2}\right)^{\frac{p_{2}}{2}} d\nu_{\alpha_{2}+p_{2}t}(x) \lesssim \|{_{s+t}}T_{\kappa}\|^{p_{2}}\cdot \|\lambda_{k}\|^{p_{2}}_{\ell^{p_{1}}}.
\end{equation}
Let ${E_{\delta}(a_{k})}$ be the associated sets to the sequence $\{a_{k}\}$ in Lemma \ref{ak}. Then we have
\begin{align*}
\sum_{k=1}^{\infty}&\left|\lambda_{k}\right|^{p_{2}}\int_{E_{\delta}(a_{k})} \left|{_{s+t}}T_{\kappa}f_{k} (x)\right|^{p_{2}}d\nu_{\alpha_{2}+p_{2}t}(x) \\
&= \int_{\mathbb{B}}\left(\sum_{k=1}^{\infty}\left|\lambda_{k}\right|^{p_{2}} \left|{_{s+t}}T_{\kappa}f_{k} (x)\right|^{p_{2}}\chi_{E_{\delta}(a_{k})}(x) \right)d\nu_{\alpha_{2}+p_{2}t}(x).
\end{align*}
If $p_{2}\geq 2$, then $\frac{2}{p_{2}}\leq 1$, and from the fact that $\ell^{2/p_{2}}$ injects continuously into $\ell^{1}$ we have
\begin{align*}
\sum_{k=1}^{\infty}&\left|\lambda_{k}\right|^{p_{2}}\int_{E_{\delta}(a_{k})} \left|{_{s+t}}T_{\kappa}f_{k} (x)\right|^{p_{2}}d\nu_{\alpha_{2}+p_{2}t}(x) \\
&\leq \int_{\mathbb{B}}\left(\sum_{k=1}^{\infty}\left|\lambda_{k}\right|^{2} \left|{_{s+t}}T_{\kappa}f_{k} (x)\right|^{2}\chi_{E_{\delta}(a_{k})} (x) \right)^{p_{2}/2}d\nu_{\alpha_{2}+p_{2}t}(x)\\
&\leq \int_{\mathbb{B}}\left(\sum_{k=1}^{\infty}\left|\lambda_{k}\right|^{2} \left|{_{s+t}}T_{\kappa}f_{k} (x)\right|^{2} \right)^{p_{2}/2}d\nu_{\alpha_{2}+p_{2}t}(x).
\end{align*}
If $0<p_{2}< 2$, then $\frac{2}{p_{2}}>1$. Thus by H\"{o}lder's inequality we get
\begin{align*}
&\sum_{k=1}^{\infty}\left|\lambda_{k}\right|^{p_{2}}\int_{E_{\delta}(a_{k})} \left|{_{s+t}}T_{\kappa}f_{k} (x)\right|^{p_{2}}d\nu_{\alpha_{2}+p_{2}t}(x) \\
&\leq \int_{\mathbb{B}}\left(\sum_{k=1}^{\infty}\left|\lambda_{k}\right|^{2} \left|{_{s+t}}T_{\kappa}f_{k} (x)\right|^{2}  \right)^{p_{2}/2}\\
&\times\left( \sum_{k=1}^{\infty} \chi_{E_{\delta}(a_{k})} (x)\right)^{1-p_{2}/2}d\nu_{\alpha_{2}+p_{2}t}(x)\\
&\leq N^{1-p_{2}/2}\int_{\mathbb{B}}\left(\sum_{k=1}^{\infty}\left|\lambda_{k}\right|^{2} \left|{_{s+t}}T_{\kappa}f_{k} (x)\right|^{2} \right)^{p_{2}/2}d\nu_{\alpha_{2}+p_{2}t}(x)
\end{align*}
since any point $x$ belongs to at most $N$ of the sets ${E_{\delta}(a_{k})}$ by Lemma \ref{ak} (iii).  Combining the last two
inequalities and applying (\ref{usingKhin}), we obtain
\begin{align*}
\sum_{k=1}^{\infty}&\left|\lambda_{k}\right|^{p_{2}}\int_{E_{\delta}(a_{k})} \left|{_{s+t}}T_{\kappa}f_{k} (x)\right|^{p_{2}}d\nu_{\alpha_{2}+p_{2}t}(x) \\
&\leq \max\{1,N^{1-p_{2}/2}\}\int_{\mathbb{B}}\left(\sum_{k=1}^{\infty}\left|\lambda_{k}\right|^{2} \left|{_{s+t}}T_{\kappa}f_{k} (x)\right|^{2} \right)^{p_{2}/2}d\nu_{\alpha_{2}+p_{2}t}(x)\\
&\lesssim \|{_{s+t}}T_{\kappa}\|^{p_{2}}\cdot \|\lambda_{k}\|^{p_{2}}_{\ell^{p_{1}}}.
\end{align*}
Since by subharmonicity (\ref{equsubharmonic2}) we have
\begin{equation*}
 |{_{s+t}}T_{\kappa}f_{k} (a_{k})|^{p_{2}}\lesssim  \frac{1}{(1-|a_{k}|^2)^{n+\alpha_{2}+p_{2}t}}\int_{E_{\delta}(a_{k})} |{_{s+t}}T_{\kappa}f_{k} (x)|^{p_{2}} d\nu_{\alpha_{2}+p_{2}t}(x),
\end{equation*}
which yields
\begin{equation}\label{*up}
\sum_{k=1}^{\infty}\left|\lambda_{k}\right|^{p_{2}} (1-|a_{k}|^2)^{n+\alpha_{2}+p_{2}t}|{_{s+t}}T_{\kappa}f_{k} (a_{k})|^{p_{2}}\lesssim  \|{_{s+t}}T_{\kappa}\|^{p_{2}}\cdot \|\lambda_{k}\|^{p_{2}}_{\ell^{p_{1}}}.
\end{equation}
Now, notice that
\begin{equation*}
 {_{s+t}}T_{\kappa}f_{k} (a_{k})= (1-|a_{k}|^2)^{n+s-\frac{n+\alpha_{1}}{p_{1}}}\int_{\mathbb{B}} |R_{s+t}(y,a_k)|^{2}(1-|y|^2)^{s+t-\alpha} d\mu(y).
\end{equation*}
Therefore, proceeding as in the case $\zeta\geq1$, we obtain
\begin{equation*}
\frac{\kappa(E_{\delta}(x))}{(1-|a_{k}|^2)^{n+s+2t+\frac{n+\alpha_{1}}{p_{1}}}}\lesssim  {_{s+t}}T_{\kappa}f_{k} (a_{k}).
\end{equation*}
Putting this into (\ref{*up}) above, we get
\begin{equation*}
\sum_{k=1}^{\infty}\left|\lambda_{k}\right|^{p_{2}} \left(\frac{\kappa(E_{\delta}(x))}{(1-|a_{k}|^2)^{\eta}}\right)^{p_{2}}\lesssim  \|{_{s+t}}T_{\kappa}\|^{p_{2}}\cdot \|\lambda_{k}\|^{p_{2}}_{\ell^{p_{1}}}
\end{equation*}
with
\begin{equation}\label{eta-equ}
\eta=n+s+t+\frac{n+\alpha_{1}}{p_{1}}-\frac{n+\alpha_{2}}{p_{2}}=(n+\gamma)\zeta.
\end{equation}
Since the conjugate exponent of $(p_{1}/p_{2})$ is $(p_{1}/p_{2})'=p_{1}/(p_{1}-p_{2})$,
by duality we know that
\begin{equation*}
\{\vartheta_{k}\}:=\left\{\left(\frac{\kappa(E_{\delta}(a_{k}))}{(1-|a_{k}|^2)^{\eta}}\right)^{p_{2}}\right\}\in \ell^{p_{1}/(p_{1}-p_{2})}
\end{equation*}
with
\begin{equation*}
\|\{\vartheta_{k}\}\|_{\ell^{p_{1}/(p_{1}-p_{2})}}\lesssim  \|{_{s+t}}T_{\kappa}\|^{p_{2}}
\end{equation*}
or
\begin{equation*}
\{\kappa_{k}\}:=\left\{\frac{\kappa(E_{\delta}(x))}{(1-|a_{k}|^2)^{(n+\gamma)\zeta}}\right\}\in \ell^{p_{1}p_{2}/(p_{1}-p_{2})}=\ell^{1/(1-\zeta)}
\end{equation*}
with
\begin{equation*}
\|\{\kappa_{k}\}\|_{\ell^{1/(1-\zeta)}}=\|\{\vartheta_{k}\}\|^{1/p_{2}}_{\ell^{p_{1}/(p_{1}-p_{2})}}\lesssim  \|{_{s+t}}T_{\kappa}\|.
\end{equation*}
By Theorem \ref{Carleson-Besov1} this means that $\kappa$ is a $(\zeta,\gamma)$-Bergman-Carleson measure.

\subsection{ (ii) Implies (i).}\label{ss-implies2}
Now suppose $(ii)$ holds, that is, $\kappa$ is a $(\zeta,\gamma)$-Bergman-Carleson measure. We divide the proof into three
cases.

\textbf{Case $1$:} Let $p_{2}>1$. For this case, let $p'_{2}$ and $\alpha'_{2}$ be two numbers satisfying
\begin{equation}\label{firstss-implies2}
\frac{1}{p_{2}}+\frac{1}{p'_{2}}=1; \qquad \frac{\alpha_{2}}{p_{2}}+\frac{\alpha'_{2}}{p'_{2}}=s.
\end{equation}
Then
\begin{equation*}
\alpha'_{2}=\left(s-\frac{\alpha_{2}}{p_{2}}\right)p'_{2}=\frac{sp_{2}-\alpha_{2}}{p_{2}-1}>-1
\end{equation*}
since $s>(1+\alpha_{2})/p_{2}-1$. By Theorem \ref{Theorem-Dual-of-bpalpha}, we know that the dual of $b^{p_{2}}_{\alpha_{2}+p_{2}t}$ can be identified with $b^{p'_{2}}_{\alpha'_{2}}$ under the pairing
\begin{equation*}
[f,g]_{b^2_{s+t}}=\int_{\mathbb{B}}f(x)\overline{g(x)} \, d\nu_{s+t}(x).
\end{equation*}

Let $f\in b^{p_{1}}_{\alpha_{1}+p_{1}t}$ and $h\in b^{p'_{2}}_{\alpha'_{2}}$. By using Fubini’s theorem and
the reproducing formula (1.5) of \cite{GKU2}, since  $\alpha'_{2}>-1$ and $\alpha'_{2}+1<p'_{2}(s+t+1)$ by the $\alpha_{2}+p_{2}t>-1$, we get that
\begin{align*}
[h,{_{s+t}}T_{\kappa}f]_{b^2_{s+t}}& =\frac{V_{\alpha_{1}}}{V_{s}} \int_{\mathbb{B}}h(y)\int_{\mathbb{B}} R_{s+t}(x,y)\overline{f(x)}d\kappa(x) \, d\nu_{s+t}(y)\\
&=\frac{V_{\alpha_{1}}}{V_{s}} \int_{\mathbb{B}}\left(\int_{\mathbb{B}} R_{s+t}(x,y) h(y) d\nu_{s+t}(y)\right) \overline{f(x)}\, d\kappa(x)\\
&=\frac{V_{\alpha_{1}}}{V_{s}} \int_{\mathbb{B}}h(x)\overline{f(x)} \, d\kappa(x).
\end{align*}
The conditions for $\zeta$ and $\gamma$ in the theorem are equivalent to
\[
\lambda=\frac{1}{p_{1}}+\frac{1}{p'_{2}}, \qquad \gamma=\frac{1}{\lambda}\left(t+\frac{\alpha_{1}}{p_{1}}+\frac{\alpha'_{2}}{p'_{2}}\right).
\]
Thus, by  Proposition \ref{product-carleson},
\begin{equation*}
|[h,{_{s+t}}T_{\kappa}f]_{b^2_{s+t}}|\lesssim \int_{\mathbb{B}}|h(x)||f(x)| \, d\kappa(x)\lesssim \|f\|_{b^{p_{1}}_{\alpha_{1}+p_{1}t}}\|h\|_{b^{p'_{2}}_{\alpha'_{2}}}.
\end{equation*}
Hence ${_{s+t}}T_{\kappa}$ is bounded from $b^{p_{1}}_{\alpha_{1}+p_{1}t}$ to $b^{p_{2}}_{\alpha_{2}+p_{2}t}$.

\textbf{Case $2$:}  $p_{2}=1$. Let $f \in b^{p_{1}}_{\alpha_{1}+p_{1}t}$. For this case, since $s>\alpha_{2}$, by Fubini's theorem and Lemma \ref{norm-kernel} we have
\begin{align}
\|{_{s+t}}T_{\kappa}f\|_{b^{1}_{\alpha_{2}+t}}&\leq \int_{\mathbb{B}}\left(\int_{\mathbb{B}}|R_{s+t}(x,y)|  |f(y)|  d\kappa(y)\right) d\nu_{\alpha_{2}+t}(x)\nonumber\\
&=\int_{\mathbb{B}} |f(y)|  \left(\int_{\mathbb{B}}|R_{s+t}(x,y)| (1-|x|^{2})^{\alpha_{2}+t} d\nu(x) \right)d\kappa(y)\nonumber\\
 &\lesssim \int_{\mathbb{B}} |f(y)|  (1-|y|^{2})^{\alpha_{2}-s}d\kappa(y).\label{case2}
\end{align}
Let $\vartheta$ be the measure defined by $d\vartheta(y)=(1-|y|^{2})^{\alpha_{2}-s}d\kappa(y)$. Since $\kappa$ is a $(\zeta,\gamma)$-Bergman-Carleson measure, we have $\widehat{\kappa}_{\gamma,\varepsilon} \in L^{1/(1-1/p_{1})}_\gamma$ for every $0<\varepsilon<1$ by Theorem \ref{Carleson-Besov1} and $\widehat{\kappa}_{\gamma,\delta} \lesssim (1-|x|^2)^{(\gamma+n)(1/p_{1}-1)}$ for every $0<\delta<1$ by Theorem \ref{Carleson-Besov2}, which are equivalent to $\widehat{\vartheta}_{\alpha_{1}+p_{1}t,\varepsilon} \in L^{1/(1-1/p_{1})}_{\alpha_{1}+p_{1}t}$ for every $0<\varepsilon<1$  and $\widehat{\vartheta}_{\alpha_{1}+p_{1}t,\delta} \lesssim (1-|x|^2)^{(\alpha_{1}+p_{1}t+n)(1/p_{1}-1)}$ for every $0<\delta<1$, respectively. Then, by Theorems \ref{Carleson-Besov1} and \ref{Carleson-Besov2}, $\vartheta$ is a $(1/p_{1},\alpha_{1}+p_{1}t)$-Bergman-Carleson measure. Thus, for any $f\in b^{p_{1}}_{\alpha_{1}+p_{1}t}$, we have
\begin{equation*}
\int_{\mathbb{B}}|f(y)| \, d\vartheta(y)\lesssim \|f\|_{b^{p_{1}}_{\alpha_{1}+p_{1}t}}.
\end{equation*}
Thus, by (\ref{case2})  it follows that $\|{_{s+t}}T_{\kappa}f\|_{b^{1}_{\alpha_{2}+t}}\lesssim \|f\|_{b^{p_{1}}_{\alpha_{1}+p_{1}t}}$ and so
${_{s+t}}T_{\kappa}$ is bounded from $b^{p_{1}}_{\alpha_{1}+p_{1}t}$ to $b^{p_{2}}_{\alpha_{2}+p_{2}t}$.

\textbf{Case $3$:}  $0<p_{2}<1$. Let $\{a_{k}\}$ be a sequence satisfying the conditions of Lemma \ref{ak}.
Then by Lemma \ref{Inequality} and \ref{R-alpha}  we have
\begin{align*}
|{_{s+t}}T_{\kappa}f(x)|&\lesssim \sum_{k=1}^{\infty}\int_{E_{\delta}(a_{k})} |R_{s+t}(x,y)| |f(y)| d\kappa(y)\\
&\lesssim \sum_{k=1}^{\infty}\int_{E_{\delta}(a_{k})} \frac{1}{[x,y]^{n+s+t}} |f(y)| d\kappa(y)\\
&\lesssim \sum_{k=1}^{\infty} \frac{1}{[x,a_{k}]^{n+s+t}}\int_{E_{\delta}(a_{k})} |f(y)| d\kappa(y)
\end{align*}
Now, by the subharmonicity in (\ref{equsubharmonic2}), for $y \in E_{\delta}(a_{k})$, we have
\begin{equation*}
 |f(y)|^{p_{1}}\lesssim  \frac{1}{(1-|a_{k}|^2)^{n+\alpha_{1}+p_{1}t}}\int_{E_{\delta}(a_{k})} |f (z)|^{p_{1}} d\nu_{\alpha_{1}+p_{1}t}(z).
\end{equation*}
From this we get
\begin{align*}
 \int_{E_{\delta}(a_{k})}&|f(y)|d\kappa(y)\\
& \lesssim  \frac{\kappa(E_{\delta}(a_{k}))}{(1-|a_{k}|^2)^{(n+\alpha_{1}+p_{1}t)/p_{1}}}\left(\int_{E_{\delta}(a_{k})} |f (z)|^{p_{1}} d\nu_{\alpha_{1}+p_{1}t}(z)\right)^{1/p_{1}}.
\end{align*}
Since $0<p_{2}<1$, this implies
\begin{align*}
|{_{s+t}}T_{\kappa}f(x)|^{p_{2}}\lesssim &\sum_{k=1}^{\infty} \frac{1}{[x,a_{k}]^{(n+s+t)p_{2}}}\frac{\kappa(E_{\delta}(a_{k}))^{p_{2}}}{(1-|a_{k}|^2)^{(n+\alpha_{1}+p_{1}t)p_{2}/p_{1}}}\\
&\times \left(\int_{E_{\delta}(a_{k})} |f (z)|^{p_{1}} d\nu_{\alpha_{1}+p_{1}t}(z)\right)^{p_{2}/p_{1}}.
\end{align*}
Therefore, since $(n+s)p_{2}>n+\alpha_{2}$, we can apply Lemma \ref{Integral-[x,y]} to obtain
\begin{align}
\|{_{s+t}}T_{\kappa}&f\|^{p_{2}}_{b^{p_{2}}_{\alpha_{2}+p_{2}t}} \nonumber\\
\lesssim &\sum_{k=1}^{\infty} \frac{\kappa(E_{\delta}(a_{k}))^{p_{2}}}{(1-|a_{k}|^2)^{(n+\alpha_{1}+p_{1}t)p_{2}/p_{1}}}\left(\int_{E_{\delta}(a_{k})} |f (z)|^{p_{1}} d\nu_{\alpha_{1}+p_{1}t}(z)\right)^{p_{2}/p_{1}}\nonumber\\
&\times \int_{\mathbb{B}}\frac{(1-|x|^{2})^{\alpha_{2}+p_{2}t}}{[x,a_{k}]^{(n+s+t)p_{2}}}d\nu(x)\nonumber\\
\lesssim &\sum_{k=1}^{\infty} \frac{\kappa(E_{\delta}(a_{k}))^{p_{2}}}{(1-|a_{k}|^2)^{(n+\alpha_{1}+p_{1}t)p_{2}/p_{1}}}\left(\int_{E_{\delta}(a_{k})} |f (z)|^{p_{1}} d\nu_{\alpha_{1}+p_{1}t}(z)\right)^{p_{2}/p_{1}}\nonumber\\
&\times (1-|a_{k}|^{2})^{n+\alpha_{2}-(n+s)p_{2}}. \label{together}
\end{align}

First, assume that  $\zeta\geq 1$. Since $\kappa$ is a $(\zeta,\gamma)$-Bergman-Carleson measure, by Theorem \ref{Carleson-Besov2} we get
\begin{equation*}
 \kappa(E_{\delta}(a_{k}))\lesssim (1-|a_{k}|^{2})^{(n+\gamma)\zeta}.
\end{equation*}
Bearing in mind (\ref{eta-equ}), this, together with (\ref{together}) and the fact that $p_{2}\geq p_{1}$ (due to the assumption $\zeta\geq 1$), yields
\begin{align*}
\|{_{s+t}}T_{\kappa}f\|^{p_{2}}_{b^{p_{2}}_{\alpha_{2}+p_{2}t}}&\lesssim \sum_{k=1}^{\infty} \left(\int_{E_{\delta}(a_{k})} |f (z)|^{p_{1}} d\nu_{\alpha_{1}+p_{1}t}(z)\right)^{p_{2}/p_{1}}\\
&\lesssim  \left(\sum_{k=1}^{\infty}\int_{E_{\delta}(a_{k})} |f (z)|^{p_{1}} d\nu_{\alpha_{1}+p_{1}t}(z)\right)^{p_{2}/p_{1}}\lesssim
\|f\|^{p_{2}}_{b^{p_{1}}_{\alpha_{1}+p_{1}t}}.
\end{align*}
Hence, ${_{s+t}}T_{\kappa}$ is bounded from $b^{p_{1}}_{\alpha_{1}+p_{1}t}$ to $b^{p_{2}}_{\alpha_{2}+p_{2}t}$.

 Next assume that $0<\zeta<1$.
Then $p_{1}>p_{2}$, and using H\"{o}lder's inequality in (\ref{together}), we get
\begin{align*}
\|{_{s+t}}T_{\kappa}&f\|^{p_{2}}_{b^{p_{2}}_{\alpha_{2}+p_{2}t}}\\
&\lesssim \sum_{k=1}^{\infty} \frac{\kappa(E_{\delta}(a_{k}))^{p_{2}}}{(1-|a_{k}|^2)^{(n+\gamma)\zeta p_{2}}}\left(\int_{E_{\delta}(a_{k})} |f (z)|^{p_{1}} d\nu_{\alpha_{1}+p_{1}t}(z)\right)^{p_{2}/p_{1}}\\
&\leq  \left\{\sum_{k=1}^{\infty}\left[ \frac{\kappa(E_{\delta}(a_{k}))^{p_{2}}}{(1-|a_{k}|^2)^{(n+\gamma)\zeta p_{2}}}\right]^{p_{1}/(p_{1}-p_{2})}\right\}^{1-p_{2}/p_{1}}\\
&\times\left(\sum_{k=1}^{\infty}\int_{E_{\delta}(a_{k})} |f (z)|^{p_{1}} d\nu_{\alpha_{1}+p_{1}t}(z)\right)^{p_{2}/p_{1}}.
\end{align*}
Since  $\kappa$ is a $(\zeta,\gamma)$-Bergman-Carleson measure, by Theorem \ref{Carleson-Besov1} we get that
\begin{align*}
\|{_{s+t}}T_{\kappa}f\|^{p_{2}}_{b^{p_{2}}_{\alpha_{2}+p_{2}t}}\lesssim &\left\{\sum_{k=1}^{\infty}\left[ \frac{\kappa(E_{\delta}(a_{k}))}{(1-|a_{k}|^2)^{(n+\gamma)\zeta }}\right]^{1/(1-\zeta)}\right\}^{(1-\zeta)p_{2}}\\
&\times\left(\sum_{k=1}^{\infty}\int_{E_{\delta}(a_{k})} |f (z)|^{p_{1}} d\nu_{\alpha_{1}+p_{1}t}(z)\right)^{p_{2}/p_{1}}\\
\lesssim &\left(\sum_{k=1}^{\infty}\int_{E_{\delta}(a_{k})} |f (z)|^{p_{1}} d\nu_{\alpha_{1}+p_{1}t}(z)\right)^{p_{2}/p_{1}}\lesssim
\|f\|^{p_{2}}_{b^{p_{1}}_{\alpha_{1}+p_{1}t}}.
\end{align*}
Hence, ${_{s+t}}T_{\kappa}$ is bounded from $b^{p_{1}}_{\alpha_{1}+p_{1}t}$ to $b^{p_{2}}_{\alpha_{2}+p_{2}t}$. The proof is complete.

\section{Proof of Theorem \ref{Theorem-2} } \label{proof2}

In this section we will prove Theorem \ref{Theorem-2}. Note again that by Theorem \ref{intertwining relation}, ${_{s,t}}T_{\mu}$ is compact from $b^{p_{1}}_{\alpha_{1}}$ to $b^{p_{2}}_{\alpha_{2}}$ if and only if  ${_{s+t}}T_{\kappa}$ is compact from $b^{p_{1}}_{\alpha_{1}+p_{1}t}$ to $b^{p_{2}}_{\alpha_{2}+p_{2}t}$.

\subsection{(i) Implies (ii).}\label{impliesii}
We divide this part into two cases as $0<\zeta<1$ and $\zeta\geq1$.

\textbf{Case $1$:} $0<\zeta<1$. In this case  a vanishing $(\zeta,\gamma)$-Bergman-Carleson measure is the same as a  $(\zeta,\gamma)$-Bergman-Carleson measures by Theorem \ref{Carleson-Besov4}. If ${_{s+t}}T_{\kappa}$ is compact, then it is bounded and by Theorem \ref{Theorem-1} we get the desired result.

\textbf{Case $2$:} $\zeta\geq1$. Since ${_{s+t}}T_{\kappa}$ is compact, then $\|{_{s+t}}T_{\kappa}f_{k}\|_{b^{p_{2}}_{\alpha_{2}+p_{2}t}}\to 0$ for any bounded sequence $\{f_{k}\}$ in $b^{p_{1}}_{\alpha_{1}+p_{1}t}$ converging to zero uniformly on compact subsets of $\mathbb{B}$. Let $\{a_{k}\}\subset \mathbb{B} $ with $|a_{k}|\to 1^{-}$ and consider the functions
\begin{equation*}
 f_{k}(x)= (1-|a_{k}|^{2})^{(n+s)-(n+\alpha_{1})/p_{1}}R_{s+t}(x,a_{k}).
\end{equation*}
Due to the conditions on $s$ and Lemma \ref{norm-kernel}, we have $\sup_{k}\|f_{k}\|_{b^{p_{1}}_{\alpha_{1}+p_{1}t}} <\infty$, and it is obvious that $f_{k}$ converges to zero uniformly on compact subsets of $\mathbb{B}$. Hence $\|{_{s+t}}T_{\kappa}f_{k}\|_{b^{p_{2}}_{\alpha_{2}+p_{2}t}}\to 0$. Therefore, proceeding as in the case $\zeta\geq1$ of that (i) Implies (ii) in Theorem \ref{Theorem-2}, for any $\delta>0$, we get
\begin{align*}
&\frac{\widehat{\kappa}_{\gamma,\delta}(a_{k})}{(1-|a_{k}|^{2})^{(n+\gamma)(\zeta-1)}}\\
&\lesssim (1-|a_{k}|^{2})^{(n+\gamma)(1-\zeta)+2(n+s+t)-(n+s)+(n+\alpha_{1})/p_{1}}{_{s+t}}T_{\kappa}f_{k}(a_{k})\\
&= (1-|a_{k}|^{2})^{t+(n+\alpha_{2})/p_{2}}{_{s+t}}T_{\kappa}f_{k}(a_{k})\\
&\lesssim \|{_{s+t}}T_{\kappa}f_{k}\|_{b^{p_{2}}_{\alpha_{2}+p_{2}t}}\to 0.
\end{align*}
Thus, by Theorem \ref{Carleson-Besov3}, the measure $\kappa$ be a vanishing $(\zeta,\delta)$-Bergman-Carleson measure.
\subsection{ (ii) Implies (i).}\label{ss-impliesi}
Now suppose $(ii)$ holds, that is, $\kappa$ is a vanishing $(\zeta,\gamma)$-Bergman-Carleson measure. We divide the proof into two
cases.

\textbf{Case $1$:}  $0<\zeta<1$. Since (ii) holds, by Theorem \ref{Theorem-1}, ${_{s+t}}T_{\kappa}$ is bounded. Also since  $0<\zeta<1$, we have $0<p_{2}<p_{1}<\infty$. Due to the space $b^{p}_{\alpha}$ is isomorphic to $\ell^{p}$ by \cite[Theorem 1.4]{DOG} for $0<p<1$ and \cite[Theorem 10.1]{GKU2} for $1\leq p<\infty$ , the result is a consequence of a general result of Banach space theory: it is known that, for $0<p_{2}<p_{1}<\infty$, every bounded operator from $\ell^{p_{1}}$ to $\ell^{p_{2}}$ is compact (see \cite[Theorem I.2.7, p. 31]{LT}).

\textbf{Case $2$:} $\zeta\geq1$. To prove that the operator ${_{s+t}}T_{\kappa}$ is compact, we must show that $\|{_{s+t}}T_{\kappa}f_{k}\|_{b^{p_{2}}_{\alpha_{2}+p_{2}t}}\to 0$  for any bounded sequence $\{f_{k}\}$ in $b^{p_{1}}_{\alpha_{1}+p_{1}t}$ converging to zero uniformly on compact subsets of $\mathbb{B}$. If $p_{2}>1$, then just like in the proof of Theorem \ref{Theorem-2}, by duality and Lemma
\ref{growth} we have (the numbers $p'_{2}$ and $\alpha'_{2}$ are the ones defined by (\ref{firstss-implies2})
\begin{align*}
\|{_{s+t}}T_{\kappa}f_{k}\|_{b^{p_{2}}_{\alpha_{2}+p_{2}t}}&\lesssim \sup_{\|h\|_{b^{p'_{2}}_{\alpha'_{2}}}\leq 1}|[h,{_{s+t}}T_{\kappa}f_{k}]_{b^{2}_{s+t}}|\\
&\leq \sup_{\|h\|_{b^{p'_{2}}_{\alpha'_{2}}}\leq 1}\int_{\mathbb{B}} |h(x)||f_{k}(x)| d\kappa(x)\\
&\lesssim \sup_{\|h\|_{b^{p'_{2}}_{\alpha'_{2}}}\leq 1}\|h\|_{b^{p'_{2}}_{\alpha'_{2}}}\int_{\mathbb{B}} |f_{k}(x)|(1-|x|^{2})^{-(n+\alpha'_{2})/p'_{2}} d\kappa(x)\\
&\lesssim \int_{\mathbb{B}} |f_{k}(x)|(1-|x|^{2})^{-(n+\alpha'_{2})/p'_{2}} d\kappa(x)
\end{align*}
Let $\vartheta$ be the measure defined by $d\vartheta(x)=(1-|x|^{2})^{-(n+\alpha'_{2})/p'_{2}} d\kappa(x)$. Since $\kappa$ is a vanishing $(\zeta,\gamma)$-Bergman-Carleson measure, using Theorem \ref{Carleson-Besov3}, we can easily see that $\vartheta$ is a vanishing $(1/p_{1},\alpha_{1}+p_{1}t)$-Bergman-Carleson measure. Thus, $\|{_{s+t}}T_{\kappa}f_{k}\|_{b^{p_{2}}_{\alpha_{2}+p_{2}t}}\to 0$.

If $0<p_{2}\leq1$, from the estimates obtained in the proof of that (ii) implies (i) in Theorem \ref{Theorem-2} (see (\ref{together})) it follows that, for any sequence $\{a_{j}\}$  satisfying the conditions of the Lemma \ref{ak}, we have
\begin{align}\label{Together2}
&\|{_{s+t}}T_{\kappa}f_{k}\|^{p_{2}}_{b^{p_{2}}_{\alpha_{2}+p_{2}t}}\nonumber \\
&\lesssim \sum_{j=1}^{\infty} \left(\frac{\kappa(E_{\delta}(a_{j}))}{(1-|a_{j}|^2)^{(n+\gamma)\zeta}}\right)^{p_{2}}\left(\int_{E_{\delta}(a_{j})} |f_{k} (x)|^{p_{1}} d\nu_{\alpha_{1}+p_{1}t}(x)\right)^{p_{2}/p_{1}}
\end{align}
Let $\epsilon>0$. Since $\kappa$ is a vanishing $(\zeta,\gamma)$-Bergman-Carleson measure, due to Theorem \ref{Carleson-Besov3}, there is $0<\delta_{0}<1$ such that
\begin{equation}\label{supaj}
 \sup_{|a_{j}|>\delta_{0}}\frac{\kappa(E_{\delta}(a_{j}))}{(1-|a_{j}|^2)^{(n+\gamma)\zeta}}<\epsilon.
\end{equation}
Split the sum appearing in (\ref{Together2}) in two parts: one over the points $a_{j}$ with $|a_{j}| \leq \delta_{0}$
and the other over the points with $|a_{j}| >\delta_{0}$. Since $\{f_{k}\}$ converges to zero uniformly
on compact subsets of $\mathbb{B}$, it is clear that the sum over the points $a_{j}$  with
$|a_{j}| \leq \delta_{0}$ (a finite sum) goes to zero as $k$ goes to infinity. On the other hand, by
(\ref{supaj}) and since $p_{2} \geq p_{1}$ (because $\zeta \geq 1$), we have
\begin{align*}
&\sum_{j:|a_{j}| >\delta_{0}}^{\infty} \left(\frac{\kappa(E_{\delta}(a_{j}))}{(1-|a_{j}|^2)^{(n+\gamma)\zeta}}\right)^{p_{2}}\left(\int_{E_{\delta}(a_{j})} |f_{k} (x)|^{p_{1}} d\nu_{\alpha_{1}+p_{1}t}(x)\right)^{p_{2}/p_{1}}\\
&\leq \epsilon^{p_{2}}\sum_{j:|a_{j}| >\delta_{0}}^{\infty} \left(\int_{E_{\delta}(a_{j})} |f_{k} (x)|^{p_{1}} d\nu_{\alpha_{1}+p_{1}t}(x)\right)^{p_{2}/p_{1}}\leq \epsilon^{p_{2}} \|f_{k}\|^{p_{2}}_{b^{p_{1}}_{\alpha_{1}+p_{1}t}}\lesssim \epsilon^{p_{2}}.
\end{align*}
Thus, $\|{_{s+t}}T_{\kappa}f_{k}\|_{b^{p_{2}}_{\alpha_{2}+p_{2}t}}\to 0$, finishing the proof.
\section{Positive Schatten Class Toeplitz Operators }\label{Positive S-C Toeplitz Op}

In this section we will prove Theorem \ref{Theorem-3}. We first briefly review the notion of Schatten class operators. If $T$ is a compact operator on a separable Hilbert space $H$ with inner product $[\cdot,\cdot]_{H}$, then there exist a non-increasing sequence $\{S_{m}(T)\}$, called the singular value sequence, and orthonormal vectors $\{e_{m}\} $ and $\{f_{m}\} $ in $H$ such that
\begin{equation*}
Tx= \sum_{m=0}^{\infty} S_{m}(T)[x,e_{m}]_{H} f_{m}
\end{equation*}
for $x\in H$. For $1\leq p<\infty$, the space of all Schatten p-class operators $S_{p}(H)$ is defined to be set of all compact operators $ T$ for which $\|T\|_{S_{p}}=\left(\sum_{m=0}^{\infty}S_{m}(T)^{p}\right)^{1/p}<\infty$. As is well known, $S_{p}(H)$ is a Banach space with the norm $\|.\|_{S_{p}}$ and is a two-sided ideal in the space of bounded linear operators on $H$. If $ T \in S_{1}(H)$ and $\{e_{m}\} $  is an orthonormal basis for $H$, then
\begin{equation*}
tr(T)= \sum_{m=0}^{\infty}[Te_{m},e_{m}]_{H},
\end{equation*}
where the series is convergent and independent of $\{e_{m}\}$. The sum above is called the trace of $T$. If $T \in  S_{p}(H)$ and $T\geq 0$, we have  $\|T\|_{S_{p}}=\left[tr(T^{p})\right]^{1/p}$ for $1\leq p<\infty$. If $T$ is a positive compact operator on $H$, then $T^{p} $ is uniquely defined, and $T \in  S_{p}(H)$ if and only if $T^{p} \in  S_{1}(H)$. See \cite{Z1}, for example, for more information and related facts.

In the rest of the paper we consider the case of a Hilbert space $H=b^{2}_{\alpha}$ for any $\alpha \in \mathbb{R}$. Given an $\alpha$, we select $s$ so as to satisfy (\ref{two}) with $p=2$, and put
\begin{equation}\label{s-alpha}
u=s-\alpha \quad \textit{and}  \quad \Phi=2s-\alpha=s+u=\alpha+2u>-1
\end{equation}
in the remaining part of the paper.  Actually it comes from self-adjointness of the operator
\[
\Lambda_{s,t} f(x) := (1-|x|^2)^{t} \int_\mathbb{B} R_{s+t}(x,y) f(y) (1-|y|^2)^t \, d\nu(y) \qquad (f\in b^2_\alpha)
\]
and the adjoint $\Lambda^{*}_{s,t}: L^{2}_{\alpha}\to L^{2}_{\alpha}$ is $\Lambda^{*}_{s,t}=\Lambda_{\alpha+t,-\alpha+s}$ whenever the operator is bounded. See \cite[Lemma 8.1]{GKU2}. So we use this notation in order to have Toeplitz operators that are direct extensions of the classical Bergman space Toeplitz operators and to have exact equalities as much as possible. See \cite[Section 3]{AK} for detail explanations in the holomorphic setting.

The norm in Definition \ref{definition of the h B-B space} yield explicit equivalent forms for the inner product of $b^{2}_{\alpha}$ as
\begin{equation}\label{inner-Pairing}
{_{t}}[f,g ]_{b^{2}_{\alpha}} =[I^{t}_{s}f,I^{t}_{s}g ]_{L^{2}_{\alpha}}=\int_{\mathbb{B}} I^{t}_{s}f(x) \, \overline{I^{t}_{s}g(x)} \ d\nu_{\alpha}(x), \qquad (f,g\in b^2_\alpha)
\end{equation}
with $t$ satisfying (\ref{alpha+pt}) with $p=2$. Also we use only the inner product $[\cdot,\cdot]_{b^{2}_{\alpha}}={_{u}}[\cdot,\cdot]_{b^{2}_{\alpha}}$. If $\alpha>-1$, it is standart to use $u=0$.

Given an $\alpha$, pick an $s$ satisfying  (\ref{two}) with $p=2$, recall that $\Phi>-1$, let $y\in \mathbb{B}$, and put
\begin{equation*}
{_{\alpha}}g_{y}(x)=\frac{R_{s}(x,y)}{\|R_{s}(\cdot,y)\|_{b^{2}_{\alpha}}} \qquad (x \in \mathbb{B}).
\end{equation*}
Since $2(n+s)-(n+\alpha)>0$, by Lemma \ref{norm-kernel}, ${_{\alpha}}g_{y}(x)\sim (1-|y|^{2})^{(n+\Phi)/2}R_{s}(x,y)$. Obviously $\|{_{\alpha}}g_{y}\|_{b^2_\alpha}=1$ for all $y\in \mathbb{B}$. Thus ${_{\alpha}}g_{y}$ is essentially a normalized
reproducing kernel; but although the kernel $R_{s}(x,y)$ is that of $b^2_s$, the normalization is
done with respect to the norm of $b^2_\alpha$, and  is considered an element of $b^2_\alpha$. It is interesting that
\begin{equation}\label{define-alphaky}
D_{s}^{u}({_{\alpha}}g_{y})(x)=\frac{R_{\Phi}(x,y)}{\|R_{\Phi}(\cdot,y)\|_{b^{2}_{\Phi}}}:={_{\Phi}}k_{y}(x),
\end{equation}
which defines ${_{\Phi}}k_{y}(x)\in b^{2}_{\Phi}$. More interesting is the fact that this family of functions in $b^2_\alpha$ spaces is the link between the bounded Toeplitz operator ${_{s,u}}T_{\mu}$ and the Berezin transforms of $\mu$. Under the conditions (\ref{s-alpha}),  $d\kappa$ takes the form
\[
d\kappa(y)=(1-|y|^{2})^{2u}d\mu(y)=(1-|y|^{2})^{\Phi-\alpha}d\mu(y),
\]
and we have
\begin{equation}\label{<stmuf,f>}
[{_{s,u}}T_{\mu}f,f]_{b^{2}_{\alpha}} =\int_{\mathbb{B}} |D_{s}^{u}f|^{2}d\kappa \qquad (f\in b^2_\alpha)
\end{equation}
which can be seen by formally exchanging the order of integrations after representing ${_{s,u}}T_{\mu}f$ as an integral, differentiation under the integral sign, and then applying the reproducing property. By (\ref{define-alphaky}) and (\ref{<stmuf,f>}), we obtain
\begin{equation}\label{<talphagy,f>}
[{_{s,u}}T_{\mu}({_{\alpha}}g_{y}),{_{\alpha}}g_{y}]_{b^{2}_{\alpha}} =\int_{\mathbb{B}} |D_{s}^{u}{_{\alpha}}g_{y}|^{2}d\kappa =\int_{\mathbb{B}} |{_{\Phi}}k_{y}|^{2}d\kappa=\widetilde{\mu}_{\Phi,\alpha,2}
\end{equation}
for those $\mu$ for which the integral converges.

We need a few lemmas before we characterize the Toeplitz operators with positive symbols that are in Schatten classes $S_{p}=S_{p}(b^{2}_{\alpha})$ for $1\leq p < \infty$.
\begin{lemma}\label{trt-berezin}
If $T: b^{2}_{\alpha}\to b^{2}_{\alpha}$ is in $S_{1}$ or positive, then
\begin{align*}
tr(T)=tr(D_{s}^{u}TD_{\Phi}^{-u})&=\int_{\mathbb{B}}[D_{s}^{u}TD_{\Phi}^{-u}R_{\Phi}(\cdot,y),R_{\Phi}(\cdot,y)]_{b^{2}_{\Phi}}d\nu_{\Phi}(y)\\
&\sim\int_{\mathbb{B}}[D_{s}^{u}TD_{\Phi}^{-u}{_{\Phi}}k_{y},{_{\Phi}}k_{y}]_{b^{2}_{\Phi}}d\nu_{-n}\\
&=\int_{\mathbb{B}}[T{_{\alpha}}g_{y},{_{\alpha}}g_{y}]_{b^{2}_{\alpha}}d\nu_{-n}
\end{align*}
\end{lemma}
where $D_{s}^{u}TD_{\Phi}^{-u}$ is the operator on $b^{2}_{\Phi}$.

\begin{proof}
Let $\{ e_{q} : q\in \mathbb{N}^{n} \}$ be an orthonormal basis for $b^{2}_{\alpha}$ with respect to the
inner product $[\cdot,\cdot]_{b^{2}_{\alpha}}$. Put $f_{q}=D_{s}^{u}e_{q}$. Then $\{ f_{q} : q\in \mathbb{N}^{n} \}$ is an orthonormal
basis for $b^{2}_{\Phi}$ with respect to the inner product $[\cdot,\cdot]_{b^{2}_{\Phi}}$ by Lemma \ref{Apply-Dst}.
Then
\begin{align*}
tr(T)=\sum_{q}[Te_{q},e_{q}]_{b^{2}_{\alpha}}=\sum_{q}[D_{s}^{u}Te_{q},D_{s}^{u}e_{q}]_{b^{2}_{\Phi}}
=\sum_{q}[(D_{s}^{u}TD_{\Phi}^{-u})f_{q},f_{q}]_{b^{2}_{\Phi}},
\end{align*}
which proves the first equality. The second equality follows by   modifying the proof
of Theorem 6.4 and Corollary 6.5 of \cite{Z1} for the ball and for weighted harmonic Bergman spaces. The inequality follows from Lemma \ref{norm-kernel}. Finally the last equality follows from (\ref{inner-Pairing}) and (\ref{define-alphaky}).
\end{proof}

\begin{lemma}
We have
\begin{equation*}
tr({_{s,u}}T_{\mu})\sim\int_{\mathbb{B}} \widetilde{\mu}_{\Phi,\alpha,2} d\nu_{-n}.
\end{equation*}
\end{lemma}

\begin{proof}
The proof follows from Lemma \ref{trt-berezin} and (\ref{<talphagy,f>}).
\end{proof}

\begin{lemma}\label{s,uTphiinsp}
If $1\leq p<\infty$ and $\phi\in L^{p}_{-n}$ then ${_{s,u}}T_{\phi}\in S^{p}$.
\end{lemma}

\begin{proof}
Let $\{ e_{q} : q\in \mathbb{N}^{n} \}$ be any orthonormal basis for $b^{2}_{\alpha}$. By Lemma \ref{trt-berezin}, we have
\begin{align*}
tr({_{s,u}}T_{\phi})=\sum_{q}[{_{s,u}}T_{\phi}e_{q},e_{q}]_{b^{2}_{\alpha}}
=\sum_{q}[{_{\Phi}}T_{\phi}f_{q},f_{q}]_{b^{2}_{Q}}=tr({_{\Phi}}T_{\phi}),
\end{align*}
where ${_{\Phi}}T_{\phi}=Q_{\Phi}M_{\phi}i$ is a classical Bergman space Toeplitz operator. So ${_{s,u}}T_{\phi}\in S^{p}$ if and only if ${_{\Phi}}T_{\phi}\in S^{p}$. This is proved in exactly same way that \cite[Proposition 7.11]{Z1}  was proved.
\end{proof}
Note that by \cite[Corollary 13.4]{GKU2} with $u$ in place of $t$,
\begin{align}\label{<sumuf,f>}
[{_{s,u}}T_{\phi}f,f]_{b^{2}_{\alpha}} &=[Q_{s}M_{\phi}I^{u}_{s}f,f]_{b^{2}_{\alpha}}=\frac{V_{\Phi}}{V_{s}}[M_{\phi}I^{u}_{s}f,I^{u}_{s}f]_{L^{2}_{\alpha}}\\
&=\frac{V^{2}_{\Phi}}{V_{\alpha}V_{s}}\int_{\mathbb{B}}\phi |D_{s}^{u}f|^{2}d\nu_{\Phi} \qquad(f\in b^2_\alpha). \nonumber
\end{align}

\begin{lemma}\label{tmu<twidehatmu}
For any $\delta>0$ there exists a constant $C_{\delta}$ with the following property: If  $\mu$ is a finite positive Borel measure such that
${_{s,u}}T_{\widehat{\mu}_{\alpha,\delta}}$ is bounded on $b^{2}_{\alpha}$, then ${_{s,u}}T_{\mu}$ is bounded on $b^{2}_{\alpha}$ and ${_{s,u}}T_{\mu}\leq C_{\delta} \, {_{s,u}}T_{\widehat{\mu}_{\alpha,\delta}}$.
\end{lemma}

\begin{proof}
Let $f\in b^{2}_{\alpha}$. We compute using (\ref{<sumuf,f>}), (\ref{hyper-volume}), Fubini theorem, (\ref{equsubharmonic2}), (\ref{<stmuf,f>}) , and obtain
\begin{align*}
[{_{s,u}}T_{\widehat{\mu}_{\alpha,\delta}}f,f]_{b^{2}_{\alpha}}
&=\frac{V^{2}_{\Phi}}{V_{\alpha}V_{s}}\int_{\mathbb{B}}\frac{\mu({E_{\delta}}(x))}{v_{\alpha}(E_{\delta}(x))} |D_{s}^{u}f(x)|^{2}d\nu_{\Phi}(x)\\
&\sim \int_{\mathbb{B}}\frac{|D_{s}^{u}f(x)|^{2}}{(1-|x|^{2})^{n-2u}}\int_{\mathbb{B}} \chi_{E_{\delta}(x)}(y)d\mu(y)d\nu(x)\\
&=\int_{\mathbb{B}}\int_{E_{\delta}(y)}\frac{|D_{s}^{u}f(x)|^{2}}{(1-|x|^{2})^{n-2u}}d\nu(x)d\mu(y)\\
&\sim\int_{\mathbb{B}}\frac{1}{(1-|y|^{2})^{n+\alpha}}\int_{E_{\delta}(y)} (1-|x|^{2})^{2u}|D_{s}^{u}f(x)|^{2}d\nu_{\alpha}(x)d\mu(y)\\
&\geq C \int_{\mathbb{B}} |D_{s}^{u}f(y)|^{2}d\kappa(y)=[{_{s,u}}T_{\mu}f,f]_{b^{2}_{\alpha}},
\end{align*}
which completes the proof of lemma.
\end{proof}

\begin{proof}[Proof of  Theorem \ref{Theorem-3}]\label{Proof of 1.4}
 (i) implies (ii). ${_{s,u}}T_{\mu}\in S_{p}$. Then $({_{s,u}}T_{\mu})^{p}$ is in $S_{1}$ so that $tr(({_{s,u}}T_{\mu})^{p})$ is finite. It follows from Lemma \ref{trt-berezin} that
 \begin{equation*}
\|{_{s,u}}T_{\mu}\|_{S_{p}}^{p}=tr(({_{s,u}}T_{\mu})^{p})=\int_{\mathbb{B}}[({_{s,u}}T_{\mu})^{p}{_{\alpha}}g_{y},{_{\alpha}}g_{y}]_{b^{2}_{\alpha}}d\nu_{-n}.
\end{equation*}
It follows from \cite[Proposition 1.31]{Z1} and (\ref{<talphagy,f>})
\begin{align*}
\|{_{s,u}}T_{\mu}\|^{p}_{S_{p}}\geq \int_{\mathbb{B}}[{_{s,u}}T_{\mu}{_{\alpha}}g_{y},{_{\alpha}}g_{y}]^{p}_{b^{2}_{\alpha}}d\nu_{-n}
= \int_{\mathbb{B}} \widetilde{\mu}_{\Phi,\alpha,2}^{p}d\nu_{-n}.
\end{align*}
Thus we have (ii).

(ii) implies (iii). Suppose $\widetilde{\mu}_{\Phi,\alpha,2}\in L^{p}_{-n}$. Fix $\delta=\delta_{0}$ where $\delta_{0}$ is the number provided by Lemma \ref{Kernel-two-sided}. By Lemma \ref{x-y-close}, (\ref{hyper-volume}), Lemma \ref{Kernel-two-sided} and Lemma \ref{norm-kernel}, we have
\begin{align*}
\widehat{\mu}_{\alpha,\delta}(y)
&\sim\frac{1}{(1-|y|^2)^{n+\alpha}}\int_{E_{\delta}(y)}d\mu(x)\\
&\sim(1-|y|^2)^{\Phi-\alpha}\int_{E_{\delta}(y)}\frac{|R_{\Phi}(x,y)|^{2}}{\|R_{\Phi}(x,y)\|_{b^{2}_{\Phi}}}d\mu(x)\\
&\leq \int_{E_{\delta}(y)}\frac{|R_{\Phi}(x,y)|^{2}}{\|R_{\Phi}(x,y)\|_{b^{2}_{\Phi}}}(1-|x|^2)^{\Phi-\alpha}d\mu(x)=\widetilde{\mu}_{\Phi,\alpha,2}.
\end{align*}
Thus we have (iii) for $\delta=\delta_{0}$. That it holds for every $0<\delta<1$ is a consequence of \cite[Lemma 3.2]{CLN1}.

(iii) implies (i). Suppose $\widehat{\mu}_{\alpha,\delta}\in L^{p}_{-n}$. Then ${_{s,u}}T_{\widehat{\mu}_{\alpha,\delta}}\in S^{p}$ by Lemma \ref{s,uTphiinsp}. By positivity and \cite[Theorem 1.27]{Z1}, $\sum_{q}[{_{s,u}}T_{\widehat{\mu}_{\alpha,\delta}}e_{q},e_{q}]^{p}_{b^{2}_{\alpha}}<\infty$ for any orthonormal set $\{ e_{q}\}$ in $b^{2}_{\alpha}$. Then $\sum_{q}[{_{s,u}}T_{\mu}e_{q},e_{q}]^{p}_{b^{2}_{\alpha}}<\infty$ too by Lemma \ref{tmu<twidehatmu}. We are done by applying \cite[Theorem 1.27]{Z1} again.

(iii) implies (iv). Lemma \ref{subharmonic-measure} with $x=a_{k}$ and $\alpha=-n$, and  Lemma \ref{x-y-close} yield
\begin{equation*}
\big(\widehat{\mu}_{\alpha,\delta}(a_{k})\big)^{p}\lesssim \int_{E_{\delta}(a_{k})} \big(\widehat{\mu}_{\alpha,\delta}\big)^{p} d\nu_{-n}.
\end{equation*}
Then
\begin{equation*}
\sum_{k=1}^{\infty}\big(\widehat{\mu}_{\alpha,\delta}\big)^{p}\lesssim \sum_{k=1}^{\infty} \int_{E_{\delta}(a_{k})} \big(\widehat{\mu}_{\alpha,\delta}(a_{k})\big)^{p} d\nu_{-n}\leq N\int_{\mathbb{B}} \big(\widehat{\mu}_{\alpha,\delta}\big)^{p} d\nu_{-n}
\end{equation*}
by Lemma \ref{ak} (iii).

(iv) implies (iii). Repeated use of Lemmas \ref{x-y-close}, \ref{ak}, equation (\ref{hyper-volume}), and that $E_{\delta/2}(x) \subset E_{\delta}(a_{k})$ for $x \in E_{\delta/2}(a_{k})$ yield
\begin{align*}
\int_{\mathbb{B}} \big(\widehat{\mu}_{\alpha,\delta/2}\big)^{p} d\nu_{-n}
&\lesssim \sum_{k=1}^{\infty} \int_{E_{\delta/2}(a_{k})} \frac{\mu(E_{\delta/2}(x))^{p}}{(1-|x|^2)^{n+p(n+\alpha)}} d\nu(x)\\
&\lesssim \sum_{k=1}^{\infty}\frac{1}{(1-|a_{k}|^2)^{n+p(n+\alpha)}}\int_{E_{\delta}(a_{k})} \mu(E_{\delta}(a_{k}))^{p}d\nu(x)\\
&\sim \sum_{k=1}^{\infty}\frac{(1-|a_{k}|^2)^{n}}{(1-|a_{k}|^2)^{n+p(n+\alpha)}}\mu(E_{\delta}(a_{k}))^{p}\\
&\sim \sum_{k=1}^{\infty}\frac{\mu(E_{\delta}(a_{k}))^{p}}{v_{\alpha}(E_{\delta}(a_{k}))^{p}}= \sum_{k=1}^{\infty} \big(\widehat{\mu}_{\alpha,\delta}(a_{k})\big)^{p}.
\end{align*}
Thus, we have $\widehat{\mu}_{\alpha,\delta/2}\in L^{p}_{-n}$. Now, by Lemma \ref{Lp-mu-delta}, we have $\widehat{\mu}_{\alpha,\delta}\in L^{p}_{-n}$.  The proof is complete.
\end{proof}


\begin{thebibliography}{1}

%
\bibitem{AK}
 D. Alpay and H. T. Kaptano\u glu, \textit{Toeplitz operators on Arveson and Dirichlet spaces}, Integr. Equ. Oper. Theory,
\textbf{58} (2007), 1--33.
%
\bibitem{ABR}
S. Axler, P. Bourdon, and W. Ramey,
\textit{Harmonic function theory},
2\textit{nd} ed., Grad. Texts in Math., vol. 137, Springer, New York, 2001.
%
\bibitem{CKL}
B. R. Choe, H. Koo, and Y. Lee, \textit{
Positive Schatten class Toeplitz operators on the ball},
Studia Math., \textbf{189} (2008), 65--90.
%
\bibitem{CKY}
B. R. Choe, H. Koo, and H. Yi,
\textit{Positive  Toeplitz operators between the harmonic Bergman spaces},
Potential Analysis, \textbf{17} (2002), 307--335.
%
\bibitem{CLN1}
B. R. Choe, Y. J. Lee, and K. Na,
\textit{Toeplitz operators on harmonic Bergman spaces},
Nagoya Math. J., \textbf{174} (2004), 165--186.
%
\bibitem{CLN2}
B. R. Choe, Y. J. Lee, and K. Na,
\textit{Positive Toeplitz operators from a harmonic Bergman space into another},
Tohoku Math. J., (2) \textbf{56} (2004), no.2, 255--270.
%
\bibitem{CR}
R. R. Coifman and R. Rochberg,
\textit{Representation theorems for holomorphic and harmonic functions in $L^p$},
Ast\'erisque, \textbf{77} (1980), 12--66.
%
\bibitem{DOG}
\"{O}. F. Do\u{g}an,
\textit{Harmonic Besov spaces with small exponents},
Complex Variables and Elliptic Equations, \textbf{65(6)} (2020), 1051--1075.
%
\bibitem{DU2}
 \"{O}. F. Do\u{g}an and A. E. \"Ureyen, \textit{Inclusion relations between harmonic Bergman-Besov and weighted Bloch spaces on the
unit ball}. Czech. Math. J., \textbf{69} (2019), 503--523.
%
\bibitem{DU1}
 \"{O}. F. Do\u{g}an and A. E. \"Ureyen, \textit{Weighted harmonic Bloch spaces on the ball}.
Complex Anal. Oper. Theory, \textbf{12(5) } (2018), 1143--1177.
%
\bibitem{DS}
A. E. Djrbashian and F. A. Shamoian,
\textit{Topics in the theory of $A^p_\alpha$ spaces},
Teubner Texts in Mathematics, 105, BSB B. G. Teubner Verlagsgesellschaft, Leipzig, 1988.
%
\bibitem{Fefferman}
C. Fefferman and E. M. Stein, \textit{$H^{p}$ spaces of several variables}.
Acta Math., \textbf{129} (1972), 137--193.
%
\bibitem{GKU1}
S. Gerg\"un, H. T. Kaptano\u glu, and A. E. \"Ureyen,
\textit{Reproducing kernels for harmonic Besov spaces on the ball},
C. R. Math. Acad. Sci. Paris, \textbf{347} (2009), 735--738.
%
\bibitem{GKU2}
S. Gerg\"{u}n, H. T. Kaptano\u glu, and A. E. \"Ureyen,
\textit{Harmonic Besov spaces on the ball},
Int. J. Math., \textbf{27} (2016), no.9, 1650070, 59 pp.
%
\bibitem{JP}
M. Jevti\'c and M. Pavlovi\'c,
\textit{Harmonic Bergman functions on the unit ball in $\mathbb R^n$},
Acta Math. Hungar., \textbf{85} (1999), 81--96.
%
\bibitem{Kuran}
 \"{U}. Kuran, \textit{ Subharmonic behaviour of $|h|^{p}$ ($p>0$, $h$ harmonic)}, J. London Math. Soc., \textbf{8} (1974), 529--538.
%
\bibitem{LT}
J. Lindenstrauss and L. Tzafriri,
\textit{Classical Banach spaces}, Lecture Notes in Math.,
338, Springer-Verlag, Berlin, 1973.
%
\bibitem{LS}
C. W. Liu and J. H. Shi,
\textit{Invariant mean-value property and $\mathcal M$-harmonicity in the unit ball of $\mathbb R^n$},
Acta Math. Sin., \textbf{19} (2003), 187--200.
%
\bibitem{LSR}
C. W. Liu,  J. H. Shi,and G. Ren
\textit{Duality for harmonic mixed-norm spaces in the unit ball
of $\mathbb{R}^{n}$},
Ann. Sci. Math. Qu´ebec., \textbf{25} (2001), 179--197.
%
\bibitem{L1}
D. H. Luecking,
\textit{Multipliers of Bergman spaces into Lebesgue spaces},
Proc. Edinburgh Math. Soc. (2), \textbf{29} (1986), 125--131.
%
\bibitem{L2}
D. H. Luecking,
\textit{Embedding theorems for spaces of analytic functions via Khinchine's inequality},
Michigan Math. J., \textbf{40} (1993), no.2, 333--358.
%
\bibitem{M}
J. Miao,
\textit{Reproducing kernels for harmonic Bergman spaces of the unit ball},
Monatsh. Math., \textbf{125} (1998), 25--35.
\bibitem{M2}
J. Miao,
\textit{Toeplitz operators on harmonic Bergman spaces},
Integr. Equ. Oper. Theory,  \textbf{27} (1997), 426--438.
%
\bibitem{O}
O. L. Oleinik,
\textit{Embedding theorems for weighted classes of harmonic and analytic functions},
J. Soviet Math., \textbf{9} (1978),  228--243.
%
\bibitem{PZ}
 J. Pau and R. Zhao,
 \textit{Carleson measures and Toeplitz operators for weighted Bergman spaces of the unit
ball}, Michigan Math. J., \textbf{64} (2015), 759--796.
%
\bibitem{R}
G. Ren,
\textit{Harmonic Bergman spaces with small exponents in the unit ball},
Collect. Math., \textbf{53} (2003), 83--98.
%
\bibitem{PE}
S. P\'{e}rez-Esteva, \textit{Duality on vector-valued weighted harmonic Bergman spaces}, Studia
Math., \textbf{118} (1996), 37--47.
%
\bibitem{KS}
K. Stroethoff,
\textit{Harmonic Bergman spaces, in Holomorphic Spaces},
 Mathematical Sciences Research Institute Publications, Vol. 33 (Cambridge University, Cambridge,
1998), pp. 51--63.
%
\bibitem{Z1}
 K. Zhu,
\textit{Operator Theory in Function Spaces},
Second Edition, Math. Surveys and Monographs, Vol. 138,
American Mathematical Society, Providence, Rhode Island, 2007.

%

\bibitem{ZY}
A. Zygmund,
\textit{Trigonometric series. Vol. I, II},
3\textit{rd} ed., Cambridge Mathematical Library, Cambridge University Press, (Cambridge, 2002).

\end{thebibliography}
\end{document}